\documentclass{article}

\usepackage{PRIMEarxiv}

\usepackage[dvipsnames]{xcolor}

\usepackage[utf8]{inputenc} 
\usepackage[T1]{fontenc}    
\usepackage{hyperref}       
\usepackage{url}            
\usepackage{booktabs}       
\usepackage{amsfonts}       
\usepackage{nicefrac}       
\usepackage{microtype}      
\usepackage{lipsum}
\usepackage{fancyhdr}       
\usepackage{graphicx}       
\graphicspath{{media/}}     

\usepackage{amsmath,amssymb,amsfonts}
\usepackage{algorithmic}
\usepackage{graphicx}

\usepackage{algorithm}
\usepackage[english]{babel}

\usepackage{comment}

\usepackage{tkz-euclide}
\usepackage{pst-plot}

\usepackage{caption}
\usepackage{subcaption}
\usepackage{multirow}

\usepackage{tikz}
\usepackage{pgfplots}
\pgfplotsset{compat = newest}

\usetikzlibrary{spy}

\usepackage{amsthm}

\newtheorem{theorem}{Theorem}[section]
\newtheorem{corollary}{Corollary}[theorem]
\newtheorem{lemma}[theorem]{Lemma}

\newtheorem{definition}{Definition}[section]

\theoremstyle{remark}
\newtheorem*{remark}{Remark}

\usepackage{comment}

\usetikzlibrary{positioning}

\newcommand{\R}{\mathbb{R}} 
\newcommand{\dd}{\mathrm{d}} 

\title{Learned RESESOP for solving inverse problems with inexact forward operator}

\author{
  Mathias S. Feinler, Bernadette N. Hahn \\
  Department of Mathematics \\
  University of Stuttgart \\
  Germany\\
  \texttt{\{mathias.feinler, bernadette.hahn\}@imng.uni-stuttgart.de} \\
  }

\begin{document}
\maketitle

\begin{abstract}
	When solving inverse problems, one has to deal with numerous potential sources of model inexactnesses, like object motion, calibration errors, or simplified data models. Regularized Sequential Subspace Optimization (ReSeSOp) allows to compensate for such inaccuracies within the reconstruction step by employing consecutive projections onto suitably defined subspaces. However, this approach relies on a priori estimates for the model inexactness levels which are typically unknown. In dynamic imaging applications, where inaccuracies arise from the unpredictable dynamics of the object, these estimates are particularly challenging to determine in advance. To overcome this limitation, we propose a learned version of ReSeSOp which allows to approximate inexactness levels on the fly. The proposed framework generalizes established unrolled iterative reconstruction schemes to inexact forward operators and is particularly tailored to the structure of dynamic problems. We also present a comprehensive mathematical analysis regarding the effect of dependencies within the forward problem, clarifying when and why dividing the overall problem into subproblems is essential. The proposed method is evaluated on various examples from dynamic imaging, including datasets from a rheological CT experiment, brain MRI, and real-time cardiac MRI. The respective results emphasize improvements in reconstruction quality while ensuring adequate data consistency.

\end{abstract}


\section{Introduction}

Solving an inverse problem means extracting a searched-for quantity $s^\mathrm{ref}\in X$ from measured data $y^\delta\in Y$ where $y^\delta=\mathcal{A}s^\mathrm{ref} + n$ with measurement noise $n$ of a certain distribution, e.g. Gaussian, and with forward operator $\mathcal{A}:X\to Y$ modelling the data acquisition process. A classic example is the task of image reconstruction in tomography. Inverse problems are in general ill-posed, i.e. in order to control the influence of the data error $n$ on the solution, suitable regularization methods need to be applied. 
Furthermore, an accurate mathematical description of the forward operator $\mathcal{A}$ is mandatory to compute a regularized solution. However, in many applications, the underlying data acquisition process can only be modelled up to a certain degree or an accurate modelling might be too computationally expensive to be used within the reconstruction step. In these cases, one has to use an idealized forward operator $\mathcal{A}^\eta$. 

In computerized tomography (CT), for instance, the standard forward model used for reconstruction is the Radon (or ray) transform which however ignores beam hardening or scattering effects. In magnetic resonance imaging (MRI), the Fourier transform is typically used as forward operator, even though the Bloch-equations describe the data generation process more accurately. Also technical imprecisions result in inexact forward models, such as, e.g. in MRI, temperature fluctuations, field inhomogeneities \cite{fessler2004}, cross talk, saturations, chemical shifts, spin history effects \cite{Gibby2000}, etc. In \cite{blanke2020}, it was further shown that unknown motion of the searched-for quantity during the data acquisition can be interpreted as an inexactness in the forward model.

This motivates that methods for computing a regularized solution are needed which can account for inaccuracies in the forward model.  

\subsection{Interpreting the inexactness as noise}

A first intuitive approach is to use the idealized forward operator $\mathcal{A}^\eta$ instead of $\mathcal{A}$ and to interpret the inexactness as additional noise term $n_\mathcal{A}$. This shifts the problem of accounting for model uncertainties to studying suitable noise models for specific applications. For MRI, appropriate noise models are for instance analyzed in \cite{macovski1996}. Assuming the additional noise term to be Gaussian, the logarithmic MAP estimator can be found by minimizing the squared data error \cite{scherzer2008} together with prior information on the image $s$, for instance sparsity in the wavelet basis for undersampled MRI \cite{lustig2007}. However, in general, this assumption will not be satisfied. Imprecise coil sensitivities or object motion, for example, are not well modelled by Gaussian noise (nor by any noise with known distribution).  

\subsection{Explicit estimation strategies to account for the inexactness}
To account for object motion in dynamic inverse problems or imprecise coil sensitivities in MRI, 
an adaption of the forward operator is possible by explicit estimation procedures. The amount of motion estimation algorithms is quite extensive and an active field of research. However, such strategies are typically developed for individual applications, individual sampling schemes and special types of dynamics \cite{spieker2023, zaitsev2015, cordero2016, feinler2023, feinler2024, burger2018} and are often computationally demanding. In addition, incorporating explicit motion information can still cause model inaccuracies (either due to their numerical computation or a simplified modelling of the dynamics). In contrast, only a few coil estimation and enhancement algorithms have been developed  \cite{allison2012, griswold2006, uecker2014, peng2022, chen2023}.

\subsection{Adapting data consistency terms}

Accounting for model inexactnesses is in particular relevant since most regularization methods rely on the concept of \emph{data consistency} of some sort. Data consistency can be enforced by introducing a cost functional
\begin{equation}
	\mathcal{C}(s) = \| y^\delta - \mathcal{A}^\eta(s) \|^2 \nonumber
\end{equation}
with suitable norm $\|\cdot\|$. If the forward operator is exactly known (i.e. $\mathcal{A}^\eta = \mathcal{A}$) and if the measured data is exact, the searched-for reconstruction should minimize the data consistency term. In case of measurement noise, however, the best solution, can only satisfy data consistency up to the noise level. 

Similarly, if the forward operator $\mathcal{A}$ is not exactly known, using a data consistency term $\mathcal{C}$ with simplified model operator can introduce artefacts and errors in the computed solution. Considering analytical methods like CG-SENSE in MRI \cite{pruessmann1999}, which incorporate such data consistency terms, we can observe, for instance, strong aliasing effects due to inaccurate or dependent coil sensitivities when high acceleration factors are employed. Similarly, ignoring the dynamic behavior of the object can result in reconstructions with strong motion artefacts, see e.g. \cite{zaitsev2015}.

Data consistency terms are also widely used in iterative solution schemes or variational regularization methods which aim at minimizing functionals of type \begin{equation}
	\mathcal{C}(s) + \gamma \, \mathcal{T}(s), \quad \gamma > 0 \nonumber
\end{equation} with additional penalty functional $\mathcal{T}$ and regularization parameter $\gamma$. This idea is further transferred to learned methods, like variational networks \cite{hammernik2018}, or NETT \cite{li2019}, where the penalty term $\mathcal{T}$ is learned. The minimization of such functionals requires in particular evaluations of the gradient of the data consistency term and the penalty term. The iteration is typically unrolled for a fixed number of iterates, where the step-size is learned as well.

The learned primal and learned primal-dual algorithms of Adler et al. \cite{adler2018} do not explictly learn the penalty term $\mathcal{T}$ but directly approximate its gradient by a neural network. Additionally, the gradient of the data consistency term is passed to a neural network allowing for more versatile adaptions than just the bare scaling with a step-size c.f. variational networks. 

All of the mentioned learned reconstruction methods can be represented by an update scheme of the following type (ignoring memory lanes)
\begin{align}
	d^{(k+1)} &= \mathcal{D}_{\Xi^{(k)}}(\mathcal{A}^\eta s^{(k)}, y^\delta, d^{(k)}) \nonumber \\
	s^{(k+1)} &= s^{(k)} + \mathcal{P}_{\theta^{(k)}}(s^{(k)}, \mathcal{A}^{\eta *} d^{(k+1)}) \nonumber
\end{align}
where $\mathcal{P}_{\theta^{(k)}}$ is the primal network with parameters $\theta^{(k)}$, $\mathcal{D}_{\Xi^{(k)}}$ is the dual network with parameters $\Xi^{(k)}$, and $d^{(k)}$ represents the dual variable with the same dimension as $y^\delta$. For the learned primal version it holds $$\mathcal{D}_{\Xi^{(k)}}(\mathcal{A}^\eta s^{(k)}, y^\delta, d^{(k)}) = y^\delta - \mathcal{A}^\eta s^{(k)} $$ and for variational networks, the primal network further reduces to $$ \mathcal{P}_{\theta^{(k)}}(s^{(k)}, \mathcal{A}^{\eta *} d^{(k+1)}) = \varphi_k \mathcal{A}^{\eta *} d^{(k+1)} + \mathcal{P}_{\theta^{(k)}}(s^{(k)})$$
with learned stepsize $\varphi_k$.
All parameters are optimized for a fixed number of iterations.

Variational Networks are the strictest analogy to variational regularization methods. These learn a regularization term for each cascade of an unrolled optimization procedure. The learned primal network additionally learns to adapt and filter gradients. Empirically, gradients are denoised and spatially filtered. Further, memory lanes are used. These allow to learn higher order optimization schemes c.f. CG. Since more than one iterate is memorized, the network can effectively \emph{choose} where to evaluate the data consistency term in order to grasp previously untouched data-encoded information. 

Since data is a realization of a certain distribution, learning in dual space can be beneficial. Especially in CT the visual appearance of sinograms allows to eliminate inconsistent measurements. The main interpretation of these networks usually comprises that missing data is completed and noise is removed. However, this can also be seen as an adaption of the true gradient to a projection step. Therefore, data is not only completed, but also altered in such ways that the backprojection step is more advantageous than just the analytic data consistency gradient. These additional degrees of freedom empirically improve reconstruction quality, with the downside of loosing strict data consistency.

\subsection{Regularized sequential subspace optimization (ReSeSOp)}

A rigorous way to account for the inexactness of forward operators in the reconstruction step provides the framework of regularized sequential subspace optimization (ReSeSOp) \cite{blanke2020, schoepfer2008}. The underlying idea of this iterative procedure is to sequentially project from some initial value onto suitable subspaces that contain the searched-for solution. In \cite{blanke2020}, it was shown that these subspaces can be chosen in accordance to the model inexactness (and the noise level) such that the sequence of iterates generated by the imperfect model $\mathcal{A}^\eta$ (and the noisy data $y^\delta$) converges to a solution of $\mathcal{A} s = y$ as noise and inexactness level tend to zero. 

In particular, the framework can take local model errors depending on the data variable into account. This is especially relevant in dynamic imaging applications. For instance, for a periodic motion, the inexactness might vary over time, or in case of local deformations, might be restricted to data points covering a specific region of the object. To this end, we split $\mathcal{A} s = y$ into subproblems 
\begin{equation}
	\mathcal{A}_i s = y_i, \quad i=1,\dots,N^\mathrm{dir} \nonumber
\end{equation} and account for the local model inexactness levels
\begin{equation}
	\|\mathcal{A}^\eta_i - \mathcal{A}_i\| \leq \eta_i, \quad i=1,\dots,N^\mathrm{dir} \nonumber
\end{equation} by combining ReSeSOp with Kaczmarz‘ method, see \cite{blanke2020} for the detailed formulation and analysis.

The main drawback of this framework is that it requires upper bounds for the (local) model inexactnesses as a priori information, i.e. good estimates for the inexactness levels $\eta_i$. This is however challenging to acquire in practice. For individual applications, e.g. dynamic image reconstruction in magnetic particle imaging (MPI), the required inexactness levels can be obtained directly from the measured data \cite{nitzsche2024}. However, this procedure relies heavily on the specific data acquisition protocoll in MPI and hence cannot be generalized.

Thus, our goal in this paper is to find a general way to estimate the inexactness levels. We observe that with the known ground truth solution, we can directly compute a sharp bound for the overall error caused by using the simplified model operator and the noisy measurements. Obviously, this is infeasible in practical applications (since the ground truth is the searched-for quantity). Nevertheless, this observation opens the way for the use of supervised learning techniques. 

The subspaces used within the ReSeSOp framework not only depend on the above described error estimates but also on prescribed \emph{search directions}. In order to ensure convergence, these should be chosen as gradient of the least squares functional \begin{equation}
	\frac{1}{2} \| \mathcal{A}^\eta s-y^\delta\|^2 \nonumber
\end{equation} evaluated at the previous iterates, cf. \cite{blanke2020}. In particular, this choice ensures data consistency up to the model inexactness and the noise level. In this article, we propose to combine these classical search directions with an additional learned regularization part and to extend ReSeSOp to a joint iterative reconstruction and inexactness estimation network.

\subsection{Outline of the paper}

The paper is organized as follows. Section \ref{section:Theory} first recalls the fundamental idea and the mathematical framework of ReSeSOp. In section \ref{section:Methods} we develop our \emph{learned ReSeSOp} approach. In particular, we discuss that the in practice unavailable inexactnesses and unknown additional regularizers make a learned algorithm necessary and derive a respective strategy. Furthermore, we discuss how redundancies in the measured data make a ReSeSOp approach with multiple subproblems mandatory. This learned ReSeSOp approach is evaluated in section \ref{section:Results} on examples from dynamic imaging applications in CT and MRI. More precisely, we present results for simulated CT data for a rheological flow experiment, the fastMRI dataset as well as a real dynamic data set from spiral heart MRI. These demonstrate that the reconstruction quality improves compared to static reconstruction schemes like the filtered back projection (FBP) in CT or CG-SENSE in MRI, and surpasses the performance of other learned reconstruction schemes like learned primal \cite{adler2018} or GANs \cite{usman2020}. Finally, we conclude our article in section \ref{section:Conclusion}.

\section{Our solution framework using regularized sequential subspace optimization}\label{section:Theory}

In this section, we derive the starting point for our learned ReSeSOp method. 

\subsection{The basic principle of ReSeSOp}
In the following, we provide a brief summary of the regularized sequential subspace optimization method (ReSeSOp) for solving inverse problems with inexact model operator and noisy data. A detailed description can be found in \cite{blanke2020} and the references therein.  

Let  $\mathcal{A}:X\to Y$ be a bounded linear operator with Hilbert spaces $X$ and $Y$ equipped with scalar products $\langle \cdot,\cdot\rangle_X$, $\langle \cdot,\cdot\rangle_Y$ and respective norms $\| \cdot \|_X$ and $\| \cdot \|_Y$. If there is no ambiguity, we will omit the indices on the norms and inner products to simplify the notation. By $\mathcal{A}^*$ we denote the adjoint of $\mathcal{A}$.

For given $y\in Y$, the set 
\begin{displaymath}
	\mathcal{M}_{\mathcal{A},y} := \lbrace s \in X \, : \, \mathcal{A} s = y \rbrace
\end{displaymath}
denotes the \emph{solution set} of the operator equation $\mathcal{A} s = y$.

The underlying idea of sequential subspace optimization is to create a sequence of iterates $s^{(k)}\in X$ that converges to a solution $s \in \mathcal{M}_{{A},y}$ by sequentially projecting a starting value $s^{(1)}\in X$ onto suitable convex subsets of $X$.

Hyperplanes and stripes play an important role in this context and are defined as follows.

\vspace*{2ex}

\begin{definition}
	Let $u \in X\setminus \lbrace 0 \rbrace$ and $\alpha, \xi \in \mathbb{R}$ with $\xi \geq 0$. We call 
	\begin{displaymath}
		H(u,\alpha) := \left\lbrace f \in X \, : \, \langle u,f \rangle = \alpha \right\rbrace
	\end{displaymath}
	a \emph{hyperplane} in $X$. The set
	\begin{displaymath}
		H_{\leq}(u,\alpha) := \left\lbrace f \in X \, : \, \langle u,f \rangle \leq \alpha \right\rbrace
	\end{displaymath} 
	below $H(u,\alpha)$ is called a \emph{half-space}. Analogously we define the half spaces $H_{\geq}(u,\alpha)$, $H_{<}(u,\alpha)$, $H_{>}(u,\alpha)$. Finally, we define the \emph{stripe}
	\begin{displaymath}
		H(u,\alpha,\xi) := \left\lbrace f \in X \, : \, \big| \langle u,f \rangle - \alpha \big| \leq \xi \right\rbrace
	\end{displaymath}
	with \emph{upper bounding hyperplane} $H(u,\alpha + \xi)$ and \emph{lower bounding hyperplane} $H(u,\alpha - \xi)$. 
\end{definition}

A direct calculation shows that for arbitrary $w\in Y$
\begin{displaymath}
	\mathcal{M}_{\mathcal{A},y} \subset H(\mathcal{A}^*w,\langle w,y \rangle).
\end{displaymath}

\subsubsection{The case of exact data and exact operator}
First, we outline the classic sequential subspace optimization algorithm for solving $\mathcal{A} s =y$ with known forward operator $\mathcal{A}$ and exactly measured data $y$. The $k$-th iteration step reads
$$ s^{(k+1)} = P_{H_k} (s^{(k)}),$$
where $P_{H_k}$ denotes the metric projection onto the intersection of hyperplanes $$ H_k := \bigcap_{j\in J_k} H(\mathcal{A}^*w^{(k)}_{j}, \langle w^{(k)}_{j}, y\rangle)$$
with chosen index set $J_k$ and chosen $w^{(k)}_{j}\in Y$. The elements $u^{(k)}_{j}:=\mathcal{A}^*w^{(k)}_{j}$ are called \emph{search directions}. The value $|J_k|$ describes the number of search directions used in the $k$-th iteration step. 

\vspace*{2ex}

In Hilbert spaces, the metric projection $P_{H(u,\alpha)}(x)$ of $x \in X$ onto a hyperplane $H(u,\alpha)$ coincides with the orthogonal projection and can be expressed as
\begin{equation}
	P_{H(u,\alpha)}(x) = x - \frac{\langle u,x \rangle - \alpha}{\| u \|^2} u. \nonumber
\end{equation}

\vspace*{2ex}

Thus, in case of one search direction $u^{(k)}:=\mathcal{A}^* w^{(k)}$ with given $w^{(k)}\in Y$ in the $k$-th iteration step, $s^{(k+1)}$ reads
\begin{equation} 
	s^{(k+1)} = s^{(k)} - \frac{\langle \mathcal{A}^*w^{(k)},s^{(k)} \rangle - \langle w^{(k)}, y\rangle}{\| u^{(k)} \|^2} u^{(k)}. \label{eq:sesop2} 
\end{equation} 

\subsubsection{ReSeSOp for noisy data and inexact operator}\label{subsec:ReSeSOp1}

Now we consider the setting where only noisy data $y^\delta$ with noise level $\delta$, i.e. $$\|y-y^\delta\| = \|n\| \leq \delta,$$ 
and an idealized version $\mathcal{A}^\eta$ of the forward operator $\mathcal{A}$ with inexactness level $\eta$, i.e. $$\|\mathcal{A}-\mathcal{A}^\eta\|\leq \eta$$ are given. 
Consider the solution set 
\begin{displaymath}
	\mathcal{M}^{\rho}_{\mathcal{A},y} :=  \mathcal{M}_{\mathcal{A},y} \cap B_{\rho}(0),
\end{displaymath}
where $B_{\rho}(0)$ is the open ball with radius $\rho > 0$ around $0$. For each $z\in \mathcal{M}^{\rho}_{\mathcal{A},y}$, it holds
\begin{align*}
	\|\mathcal{A}^\eta z-y^\delta\| &= \|(\mathcal{A}^\eta-\mathcal{A})z + A z -y^\delta\|\\
	&\leq \|\mathcal{A}^\eta-\mathcal{A}\|\,\|z\| + \|\mathcal{A}z-y^\delta\|\\
	&\leq \eta\,\rho + \delta.
\end{align*}
Thus, for arbitrary $w\in Y$, the solution set $\mathcal{M}^{\rho}_{\mathcal{A},y}$ is contained in the stripe 
\begin{align*}
	H^{\delta,\eta,\rho} &:= H\big((\mathcal{A}^{\eta})^*w, \langle w,y^\delta\rangle, (\delta + \eta\rho)\| w \|\big) \\
	&= \left\lbrace x \in X \, : \, \left\|\left\langle (\mathcal{A}^{\eta})^*w,x \right\rangle - \left\langle w,y^\delta \right\rangle \right\| \leq (\delta + \eta\rho) \| w \| \right\rbrace,
\end{align*}
cf. \cite{blanke2020}. This motivates to replace the metric projection onto hyperplanes in \eqref{eq:sesop2} by the metric projection onto stripes whose width is chosen in accordance to the noise and inexactness level.

The metric projection onto a stripe corresponds to the metric projection onto the respective bounding hyperplane. Since both an upper bounding hyperplane and a lower bounding hyperplane exist, it is essential to determine which one to select for the metric projection. Either option satisfies the required level of inexactness. However, the upper bounding hyperplane is positioned closest to the current iterate. This proximity results in smaller adjustments to the iterate. In the light of accounting for local inexactnesses, cf. section \ref{subsec:local_model_inexactnesses}, the choice of the upper bounding hyperplane further minimizes the interference with other subproblems. 

Altogether, we obtain the iteration scheme
\begin{equation} 
	s^{(k+1)} = s^{(k)} -  \frac{\langle (\mathcal{A}^\eta)^*w^{(k)},s^{(k)} \rangle - \langle w^{(k)},y^\delta\rangle- (\delta + \eta\rho)\| w^{(k)} \|}{\| u^{(k)} \|^2} u^{(k)} \label{eq:ReSeSOp1suchr}
\end{equation}
with chosen $w^{(k)} = w(s^{(k)})\in Y$ and $u^{(k)} = u(s^{(k)}):=(\mathcal{A}^\eta)^* w^{(k)}$. In particular, the choice $$w(s^{(k)}):= \mathcal{A}^\eta s^{(k)}-y^\delta$$ ensures convergence and - in combination with the discrepancy principle - the regularization property of the iteration scheme \eqref{eq:ReSeSOp1suchr} which has been studied in detail in \cite{blanke2020}. With this particular choice for $w^{(k)}$, we can further simplify the expression on the right-hand side yielding
\begin{equation}
	s^{(k+1)} = s^{(k)} -  \frac{\| w(s^{(k)}\|\, \left(\| w(s^{(k)}) \| - (\delta+\eta\rho)\right) }{\| u(s^{(k)}) \|^2} u(s^{(k)}). \label{eq:projection_simplified}
\end{equation}

\subsection{Application in dynamic imaging and accounting for local model inexactnesses}\label{subsec:local_model_inexactnesses}

In many practical applications, collecting the measurement $y$ takes a considerable amount of time. In computerized tomography, for instance, it takes time to rotate the radiation source around the investigated object. In MRI, the generation of spin echos and the encoding of k-space trajectories can stress even longer acquisition times. Within the standard formulation of static inverse problems $\mathcal{A}^\mathrm{stat}f=y$ (please note that we write here $f$ as symbol for the searched-for quantity), 
this can be expressed by specifying the data variable, i.e., the equation 
more precisely reads 
\begin{equation}
	\mathcal{A}^\mathrm{stat} f (t,\nu) = y(t,\nu)\quad \textrm{for all} \ (t,\nu)\in [0,T]\times \Omega_Y \nonumber
\end{equation}
with time interval $[0,T]\subset \mathbb{R}$ covering the time period required for collecting the data, and $\Omega_Y$ could be considered, e.g., as bounded subset of $\mathbb{R}^m$. 
To explicitly indicate the dependence on $t$, we write
$$ \mathcal{A}_t^{\mathrm{stat}} f (\cdot) := \mathcal{A}^{\mathrm{stat}} f (t,\cdot).$$
In MRI, for instance, the forward model of the static problem is given by
\begin{equation}
	\mathcal{A}_{t,c}^{\mathrm{stat}}[ \cdot] := M_t \mathcal{F}[S_c \cdot], \label{eq:MRI_forward}
\end{equation}
where $\mathcal{F}$ is the Fourier transform, $\{S_c\}_{c=1}^{N^\mathrm{c}}$ are a priori estimated coil sensitivities and $M_t$ describes the sampling pattern at time $t\in[0, T]$. 
In case of CT, the static operator is considered as
\begin{equation}
	\mathcal{A}_{t}^{\mathrm{stat}}[ \cdot] := M_t \mathcal{R}[ \cdot], \label{eq:CT_forward}
\end{equation} 
where $\mathcal{R}$ is the Radon transform \cite{natterer2001} and $M_t$ encodes the angular sampling pattern at time $t$.

In many applications, the searched-for quantity $f$ changes during the time of the data acquisition, for instance in medical imaging due to patient or organ motion or in non-destructive testing when visualizing fluid flow in porous media. 

Typically, the individual states of the specimen show a strong temporal correlation which can be expressed by a reference configuration $s$ and a deformation model $\Gamma : [0, T] \times \mathbb{R}^n \to \mathbb{R}^n$. Thus, the exact forward operator of the dynamic inverse problem for recovering $s$ from measured data $y$ is given by
$$ \mathcal{A} s (t,\nu) := \mathcal{A}^\mathrm{stat} [s\circ \Gamma(t,\cdot)] (t,\nu).$$
In particular, the true forward model $\mathcal{A}$ depends on the unknown deformation $\Gamma$ and is therefore in practice not exactly known. Thus, only a simplified forward operator $\mathcal{A}^\eta$ can be used within the reconstruction step. If no further information on the dynamic behavior is available, the intuitive idea is to use as simplified forward model the operator from the underlying static problem, i.e. $\mathcal{A}^\eta = \mathcal{A}^\mathrm{stat}$. In order to avoid artefacts in the reconstruction, the solution scheme should account for this model inexactness which can be achieved by the ReSeSOp algorithm presented in Section \ref{subsec:ReSeSOp1}. 

\vspace*{2ex}

However, in order to work efficiently, we need to extend the method to take local inaccuracies into account: In most cases, a global error $\|\mathcal{A}-\mathcal{A}^\eta\|\leq \eta$ will quantify the model deviation too pessimistically. As an example, consider the case of periodic motion, such as respiratory or cardiac motion. Due to the periodicity, it might hold $\mathcal{A}_t s =\mathcal{A}_t^\eta s$ for certain time instances, i.e., locally, the model error is small. Another significant example is that motion estimation strategies, e.g. \cite{cordero2016,feinler2023,feinler2024}, reduce motion artefacts particularly well in fully sampled regions, while the deformations in undersampled regions in the periphery are (usually) poorly approximated. Hence, suitable reconstruction algorithms have to adapt to locally imprecise motion estimates. Thus, we want to consider a framework that takes local model errors depending on the data variable into account. 

For this purpose, we follow the suggestion of \cite{blanke2020} and focus on a set of subproblems
\[\mathcal{A}_{i} s = y_{i}, \quad i = 1, \dots, N^\mathrm{dir}.\]
This allows to assign individual inexactness levels $\eta_i$ to the corresponding fractions of data $y_i$, $i=1,\dots,N^\mathrm{dir}$, i.e.
$$ \| \mathcal{A}_i - \mathcal{A}_i^\eta\| \leq \eta_i, \quad i=1,\dots,N^\mathrm{dir}.$$
Combining ReSeSOp with the classical Kaczmarz' method results in the iteration scheme presented in algorithm \ref{alg:classical_ReSeSOp}. 

\begin{algorithm}[h]
	\caption{$k$-th iteration of classical ReSeSOp}
	\label{alg:classical_ReSeSOp}
	\begin{algorithmic}[1]
		\STATE{Input: $s^{(k)}$}
		\STATE{Output: $s^{(k+1)}$}
		\STATE{Initialize: $\tilde{s}_1:=s^{(k)}$}
		\FOR{ $i=1,\ldots , N^\mathrm{dir}$}
		\STATE{$w_i = w(\tilde{s}_i):=\mathcal{A}^\eta_i \tilde{s}_i-y_i^\delta\in Y$}
		\STATE{$u_i = u(\tilde{s}_i):=(\mathcal{A}^\eta)_i^* w_i $} 
		\STATE{$\tilde{s}_{i+1} = \tilde{s}_i - \frac{\langle (\mathcal{A}^\eta)_i^*w_i,\tilde{s}_i \rangle - \langle w_i,g_i^\delta\rangle- (\delta + \eta_i \rho)\| w_i \|}{\| u_i \|^2} \, u_i$} 
		\ENDFOR
		\STATE{$s^{(k+1)} := \tilde{s}_{N^\mathrm{dir}+1}$}
	\end{algorithmic}
\end{algorithm}

The convergence of the sequence and its regularizing property (if combined with the discrepancy principle) follows directly from the general theory in \cite{blanke2020}. 

In particular, $u_i$ can be seen as a search-direction associated to the $i$-th subproblem $\mathcal{A}^\eta_i s = y^\delta_i$ and $(\tilde{s}_i)_{i=2,\dots,N^\mathrm{dir}}$ as intermediary iterates which are obtained by performing a metric projection for each individual subproblem.

\section{Learned ReSeSOp}\label{section:Methods}

The computation of the iterates requires in particular knowledge of estimates for the noise level $\delta$ and the inexactness levels $\{\eta_i\}_{i=1,\dots,N^\mathrm{dir}}$. If the exact solution $s^\mathrm{ref}$ of the unperturbed inverse problem $\mathcal{A}s=y$ was known, then we could compute
\begin{equation}
	\| \mathcal{A}_i^\eta s^\mathrm{ref}-y_i^\delta \| =: \mathcal{E}_i. \label{eq:inexactnesses}
\end{equation} 
A lower value for the residuals $\|\mathcal{A}_i^\eta \cdot - y_i^\delta\|$ cannot be expected from any computed solution. This motivates to use stripes
$$H^{\mathcal{E}_i}:= H((\mathcal{A}^\eta)^*w, \langle w,y^\delta \rangle, \mathcal{E}_i \|w\|) $$
within the ReSeSOp framework. This corresponds to finding the solution $z \in M_{\mathcal{A},y}^\rho$ with $\|\mathcal{A}_i^\eta z - y_i^\delta\|=\mathcal{E}_i$ for all $i$. 

However, precomputing $\mathcal{E}_i$ requires the ground truth solution which is of course not available in practice. In contrast, if using a supervised machine learning approach with ground truth solutions being part of the training data set, we could indeed compute such sharp inexactnesses. This is our motivation to use a suitable network architecture to learn the required inexactness levels from a database.

The original ReSeSOp algorithm uses search directions one after another. This requires many (intermediary) iterations and is computationally demanding. A more efficient approach is to include all search directions at once, i.e. to compute 
\begin{equation}
	s^{(k+1)} = s^{(k)} - \sum_{i=1}^{N^\mathrm{dir}} \kappa^{(k)}_i u_i(s^{(k)}) \label{eq:all_projection}
\end{equation} 
with suitable stepsizes $\{\kappa_i^{(k)}\}_{i=1}^{N^\mathrm{dir}}$ such that $\|\mathcal{A}^\eta_i s^{(k+1)} - y_i^\delta\| = \mathcal{E}_i$ for all $i$.

This is equivalent to the projection onto the intersection of $N^\mathrm{dir}$ stripes. Hence, in light of the original literature \cite{blanke2020}, we speak in this context of $N^\mathrm{dir}$ search directions in the following.

However, the computation of $\kappa^{(k)}_i$ is in general non-trivial and requires the knowledge of all inexactness levels $\{\mathcal{E}_j\}_j$, all search directions $\{u^{(k)}_j\}_j$ and all residuals $\{w^{(k)}_j\}_j$. 

By simple rearrangements, we can show that the computation of optimal $\kappa_i$ is equivalent to solving the following quadratic system of equations
\begin{equation}
	0 = a_i + \kappa^\intercal b_i + \kappa^\intercal C_i \kappa \quad \forall i \label{eq:quadratic_optimality_form}
\end{equation}
with $$a_i = \|w_i\|^2- \mathcal{E}_i^2, \quad b_i = -2\left(\langle w_i, \mathcal{A}^\eta_i u_j\rangle \right)_j, \quad C_i = \left(\langle \mathcal{A}^\eta_i u_j, \mathcal{A}^\eta_i u_l\rangle \right)_{j,l}.$$ 

In general, this problem is intricate to solve. Since we rely on iterative methods we do not need to find the exact solution but only need to converge sufficiently fast. Hence, we want to assign this task to a neural network. Note that convergence speed of iterative solution schemes can typically be increased by overrelaxation. The extend of optimal overrelaxation is usually not known. Hence, we leave the choice of intermediary $\kappa_i^{(k)}$ free for the network.

\subsection{ReSeSOp with an additional regularizer}\label{subsec:ReSeSOp_with_additional_reg}
Due to its construction, the ReSeSOp algorithm with exact inexactnesses c.f. equation \eqref{eq:inexactnesses} searches for a solution of $\mathcal{A} s = y^\delta$ within the set 
$$\mathcal{D} := \lbrace s\in X\cap B_\rho(0) \, : \, \|\mathcal{A}_i^{\eta}s-y_i^\delta\| = \mathcal{E}_i \quad \forall i \rbrace. $$ This set however can contain predominantly images with artefacts. Therefore, we propose to include an additional regularization within each iteration step. Formally, this can be repesented by adding one additional subproblem of type $$ {R} s = 0$$ with a prescribed regularization operator ${R}$. This corresponds to augmenting the forward operator and the data via
\begin{equation}\label{eq_augmentedsystem}
	\tilde{\mathcal{A}} = \left(\begin{array}{c}
		\mathcal{A}^\eta\\
		R
	\end{array} \right) \qquad 
	\tilde{y} = \left(\begin{array}{c}
		y^\delta \\
		0
	\end{array} \right)
\end{equation}
which allows us to find a \emph{$R$-regularized solution} within the ReSeSOp solution space $\mathcal{D}$, i.e. a solution $s^*$ with $Rs^*=0$, respectively $\|Rs^*\|\leq \mathcal{E}^\mathrm{R}$ with additional prescribed inexactness level $\mathcal{E}^\mathrm{R}$.

This inexactness parameter $\mathcal{E}^\mathrm{R}$ can be seen as an additional information on the solution. For instance, in imaging applications for material testing, we can use total variation and describe the \emph{total length of edges} within the material by $\mathcal{E}^\mathrm{R}$. Similarly, \emph{perfect} regularizers are zero for artefact-free images and assume large values for images with artefacts. For such regularizers, $\mathcal{E}^\mathrm{R}=0$ is the appropriate choice in order to ensure artefact-free reconstructions. 

In particular, system \eqref{eq_augmentedsystem} with respective model inexactness fits exactly into the classical ReSeSOp framework and the convergence properties studied in \cite{blanke2020} still apply.

A perfect regularizer $R$, however, depends on the space of ground truth images. Since these spaces are (apart from academic examples) analytically barely describable, we want to use a neural network to learn the appropriate regularization term. This idea follows \cite{li2019} which proposed to use learned regularizers within a variational framework and provided a respective convergence analysis.

\subsection{Learned ReSeSOp algorithm}

In this section we derive our learned ReSeSOp algorithm. This algorithm aims to learn the stepsizes $\{\kappa_i^{(k)}\}_{i=1}^{N^\mathrm{dir}}$ for each iteration step $(k)$ as well as regularizations $R^{(k)}$ to form the reconstruction of a predicted image $s^\mathrm{pred}$ as the final result. In order to receive an end-to-end trainable neural network, we unroll the iterative procedure for a fixed number of iterations $N^\mathrm{iter}$. Due to equation \eqref{eq:inexactnesses} we can also calculate predictions for the inexactness levels $\mathcal{E}_i$ by replacing $s^\mathrm{ref}$ with $s^\mathrm{pred}$. 

From our previous considerations, the respective network receives the current iterate $s^{(k)}$ and the respective search directions $\{u_i\}_{i=1}^{N^\mathrm{dir}}$ as inputs. Additionally, we provide the result of equation \eqref{eq:inexactnesses} to assess dependencies, where we replace $s^\mathrm{ref}$ with the current iterate $s^{(k)}$. Further, we want this network to feature memory lanes. In analogy to Adler et al. \cite{adler2018}, we provide $N^\mathrm{mem}=5$ memory lanes. In the same sense, we feed old stepsizes $\kappa_i^{(k-1)}$ into the network to enhance the calculation of new stepsizes $\kappa_i^{(k)}$.

The details of the procedure are specified in Algorithm \ref{alg:learned ReSeSOp}. This algorithm can be extended to learn in frequency space by adding a dual network in a similar fashion as in \cite{adler2018}.

\begin{algorithm}
	\caption{learned ReSeSOp}
	\label{alg:learned ReSeSOp}
	\begin{algorithmic}[1]
		\STATE{$\mathrm{Input}:  \, \, y_i, \,\, i = 1,\ldots,N^\mathrm{dir}$} 
		\STATE{$\mathrm{Output}:  \, \, s^\mathrm{pred} = s^{(N^\mathrm{iter}+1)} , \,\, \mathcal{E}_i^{(N^\mathrm{iter}+1)}, \,\,  i = 1,\ldots,N^\mathrm{dir}$}
		\STATE{$\mathrm{Initialize}: s^{(1)} = \mathrm{CG}(y), \,\, \kappa_i^{(0)} = 1, \,\, i=1,\ldots,N^\mathrm{dir}$}

		\FOR{ $k=1,\ldots,N^\mathrm{iter}$}
		\FOR{ $i=1,\ldots,N^\mathrm{dir}$}
		\STATE{$w_i^{(k)} = \mathcal{A}_i^{\eta}s^{(k)} - y_i$} 
	
	\STATE{$u_i^{(k)} = \mathcal{A}_i^{\eta *}w_i^{(k)}$} 

\STATE{$\mathcal{E}_{i}^{(k)} = \|\mathcal{A}_i^{\eta}s^{(k)} - y_i\| $} 
\ENDFOR

\STATE{$\Big(R^{(k)}, \{\kappa^{(k)}_i\}_i \Big)= \mathrm{NN}_{\theta^{(k)}}(s^{(k)}, \{\kappa_i^{(k-1)}\}_i, \{\mathcal{E}_i^{(k)}\}_i, \{u_i^{(k)}\}_i)$}

\STATE{$s^{(k+1)} = s^{(k)} - R^{(k)} - \sum_i \kappa_i^{(k)} u_i^{(k)}$}     

\ENDFOR
\end{algorithmic}
\end{algorithm}

For the numerical results in this paper, we fix $N^\mathrm{iter}=8$ and $N^\mathrm{mem}=5$, and use two $\mathrm{CG}$ steps for initialization of $s^{(1)}$. The choice of initialization allows that $s^{(1)}$ is meaningful, but not fully converged, such that the search directions $u_i^{(1)}$ are not zero. Alternatively we can initialize $s^{(1)}$ with zero, which however requires more learned ReSeSOp steps to reach the same accuracy.

The number of search directions $N^\mathrm{dir}$ depends on the specific application. For dynamic scenarios, it can be sensible to split the operator into subproblems over time. In the same sense, we can split the operator into contributions per receiver coil for parallel MRI. In some cases, the structure of the operator allows a splitting into contributions over space. E.g. for MPI, we can split the operator into spatial contributions to account for inhomogeneities of the system matrix \cite{nitzsche2022}.

The Network $\mathrm{NN}_{\theta^{(k)}}$ is a U-Net \cite{ronneberger2015} with an appended encoder, parametrised by $\theta^{(k)}$. The first channel of the U-Net output is used as $R^{(k)}$. All U-Net output channels are convoluted and pooled down to the dimension of the inexactnesses. This is shown in figure \ref{fig:learned_ReSeSOp}. In practical applications, the hyperparameters of this architecture should be adapted to the actual size of the images and the number of considered subproblems. 

\begin{figure}
\centering
\includegraphics[width=\textwidth]{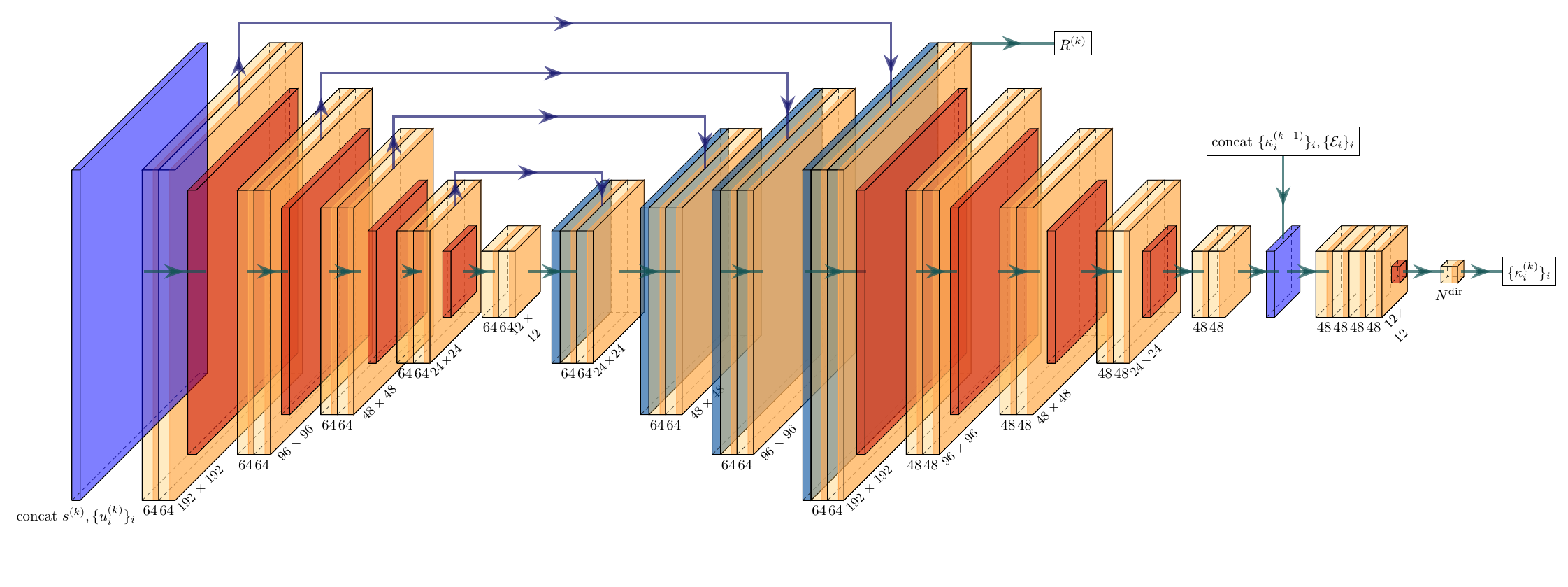}
\caption{$\text{NN}_\theta$ with parameters $\theta$ for images of resolution $N\times N=192 \times 192$. The outputs are the regularization term $R^{(k)}$ and stepsizes $\{\kappa_i^{(k)}\}_i$. The inputs are the iterate $s^{(k)}$, old stepsizes $\{\kappa_i^{(k-1)}\}_i$, current inexactness levels $\{\mathcal{E}_i^{(k)} = \|\mathcal{A}^\eta_i s^{(k)} - y_i\|\}_i$ and search directions $\{u_i^{(k)}\}_i$.}
\label{fig:learned_ReSeSOp}
\end{figure}

\subsection{Training}
In order to ensure that the prediction is close to the ground truth reference, we use the training loss
\begin{equation}
\mathcal{L}_x = \mathrm{SSIM} (s^\mathrm{ref}, s^\mathrm{pred}),
\end{equation}
where the structural similarity index (SSIM) can be replaced by other similarity metrics such as MSE, PSNR or further. As explained in the beginning of section \ref{section:Methods} the optimal stepsizes $\kappa_i^{(k)}$ are unknown for the regularized iterative procedure. Hence, we cannot restrict these by means of a loss function. However, optimal inexactnesses are known for the reconstruction $s^\mathrm{pred}$. These correspond to 
\begin{equation}
\mathcal{E}_i = \|\mathcal{A}^\eta_i s^\mathrm{ref} - y_i^\delta\| \nonumber,
\end{equation}
cf. equation \eqref{eq:inexactnesses}. 
Hence we can train the network using the additional training loss contribution
\begin{equation}
\mathcal{L}_\mathcal{E} = \sum_i (\mathcal{E}_i - \| y_i^\delta - \mathcal{A}^\eta_i s^\mathrm{pred} \| )^2. \label{eq:LossE}
\end{equation}
If $s^\mathrm{pred}$ satisfies $\|\mathcal{A}_i^\eta s^\mathrm{pred}-y_i^\delta\| = \mathcal{E}_i$ for all $i$, this means, in particular, that $s^\mathrm{pred}$ is consistent with the data and the model in the sense of ReSeSOp, i.e. up to the noise and model inexactness level. The training loss $\mathcal{L}_\mathcal{E}$ hence enforces the final solution to be contained in the manifold defined by the solution space $\mathcal{D} = \{s: \,\, \| w_i(s) \| = \mathcal{E}_i , \,\, i=1,\ldots ,N^\mathrm{dir}\}$.

We finally train the network end-to-end using Adam optimizer for several dozen epochs until convergence.

\subsection{ReSeSOp for problems with and without redundancies} \label{subsec:ReSeSOp_w_wo_redundncies}
In this subsection, we want to emphasize the importance of decomposing the full problem into subproblems and state which operator characteristics affect the necessity to do so.

\subsubsection{Orthogonal suboperators}

As we discussed before, the forward (imaging) operator $\mathcal{A}$ can be split into several suboperators $\mathcal{A}_i, \, i=1,\dots N^\mathrm{dir}$ (corresponding to the respective subproblems). 

\begin{definition}\label{def:orthogonal_suboperators}
Two suboperators $\mathcal{A}_i$ and $\mathcal{A}_j$ are called \emph{orthogonal} if
\begin{displaymath}
0 = \mathcal{A}_i \mathcal{A}_j^*.
\end{displaymath}
\end{definition}

\begin{lemma}
If all suboperators are pairwise orthogonal, we can treat subproblems individually.
\end{lemma}
\begin{proof}
If we insert definition \ref{def:orthogonal_suboperators} into \eqref{eq:quadratic_optimality_form}, we receive
\begin{displaymath}
\mathcal{E}_i^2 = \| w_i \|^2 - 2\kappa_i \langle w_i, \mathcal{A}_i u_i \rangle + \kappa_i^2 \langle \mathcal{A}_i u_i, \mathcal{A}_i u_i \rangle
\end{displaymath}
or equivalently
\begin{equation}
\mathcal{E}_i^2 = \| y_i - \mathcal{A}_i(s^{(k)} + \kappa_i u_i) \|^2. \label{eq:simplified_quadratic_optimality_form}
\end{equation}
Hence, the determination of the stepsize $\kappa_i$ is independent from the determination of stepsizes $\kappa_j$ for all $j \not= i$. Due to the orthogonality, the update corresponding to subproblem $i$ does not interfere with the update corresponding to subproblem $j \not= i$. This means if $\mathcal{E}_i = \| y_i - \mathcal{A}_i(s^{(k)}) \|$, then also $s^{(k+1)} = s^{(k)} + \kappa_j u_j$ satisfies the inexactness w.r.t subproblem $i$ since
\begin{displaymath}
\| y_i - \mathcal{A}_i(s^{(k)} + \kappa_j u_j) \| = \| y_i - \mathcal{A}_is^{(k)} - \kappa_j \mathcal{A}_i \mathcal{A}_j^* w_j\| = \| y_i - \mathcal{A}_i s^{(k)} \|^2.
\end{displaymath}
\end{proof}

\begin{lemma}
If all suboperators are pairwise orthogonal and if further 
\begin{equation}
w_i\langle \mathcal{A}_i \mathcal{A}_i^*w_i, w_i \rangle = \| w_i \|^2\mathcal{A}_i u_i, \label{eq:AiAistar}
\end{equation}
then the metric projection is the optimal update scheme and allows ReSeSOp to converge within one step. Equation \eqref{eq:AiAistar} holds in particular for all $w_i$ if $\mathcal{A}_i \mathcal{A}_i^* = I$.
\end{lemma}
\begin{proof}
For orthogonal operators it is sufficient to find stepsizes $\kappa_i$ which satisfy equation \eqref{eq:simplified_quadratic_optimality_form}. Therefore, we test our candidate the metric projection given in equation \eqref{eq:projection_simplified} and receive
\begin{equation}
\| y_i - \mathcal{A}_i(s^{(k)} + \kappa_i u_i) \|^2 = \| w_i - \frac{\|w_i\|^2 - \mathcal{E}_i \|w_i\|}{\|u_i\|^2}\mathcal{A}_iu_i \|^2.\label{eq:Ei_equals_}
\end{equation}
By using
\begin{displaymath}
w_i\|u_i\|^2 = w_i\langle \mathcal{A}_i \mathcal{A}_i^*w_i, w_i \rangle,
\end{displaymath}
together with the prerequisite, the right-hand side of \eqref{eq:Ei_equals_} reduces to $\mathcal{E}_i^2$. 
This means the metric projection is the optimal update scheme. 
\end{proof}

According to the lemma and its proof, the optimal stepsizes in \ref{eq:quadratic_optimality_form} could be even computed explicitly if the inexactness levels were known exactly. 

There exist individual imaging applications, which indeed result in orthogonal suboperators, for instance in cartesian MRI with only one single receive coil. However, in general, orthogonality of suboperators is a rather rare occurrence in imaging problems. 

Therefore, we next discuss a weaker condition on suboperators, namely \emph{linear independence}, which occurs, for instance, in sparse MRI with multiple receiver coils or subsampled CT. In this case, data is acquired without redundancy.

\subsubsection{Linearly independent subproblems} 

For the following analysis, we consider the discrete case with $\mathcal{A} \in \mathbb{C}^{M\times N}$.

\begin{lemma}\label{lem:lin_indep_rows}
If all rows of $\mathcal{A}$ are linearly independent, each $u_i$ can be extracted from the gradient $u^\Sigma$ of the full problem
\begin{displaymath}
u^\Sigma := \mathcal{A}^*(\mathcal{A} s - y) = \sum_{i=1}^{N^\mathrm{dir}} u_i.
\end{displaymath}
\end{lemma}
\begin{proof}
Since $\mathcal{A}^*$ has linearly independent columns, the Moore-Penrose inverse exists. We hence can compute 
\begin{displaymath}
(\mathcal{A}\mathcal{A}^*)^{-1} \mathcal{A} u^\Sigma = \mathcal{A} s - y.
\end{displaymath}
In data space, the separation into subproblems is just the respective ordering, i.e. 
\begin{displaymath}
\mathcal{A} s - y = (w_1, \ldots , w_{N^\mathrm{dir}})^\intercal.
\end{displaymath}
Now we can reapply the adjoint of the $i$-th subproblem on the $i$-th component respectively to get
\begin{displaymath}
\mathcal{A}^*_i (\mathcal{A}_i s - y_i) = u_i.
\end{displaymath}
Hence, we can extract each $u_i$ from $u^\Sigma$ as asserted. 
\end{proof}

The proof reveals, that the statement of lemma \ref{lem:lin_indep_rows} holds true under slightly weaker conditions. 

\begin{corollary}
If each line of a suboperator is linearly independent of all lines of other suboperators, we can extract all $u_i$ from the gradient of the full problem $u^\Sigma := \mathcal{A}^*(\mathcal{A} s - y) $
\end{corollary}
Note, that each individual suboperator can consist of linearly dependent lines. In this case, the residual value $w_i = w_i^\mathrm{ker} + w_i^\mathrm{im}$ can only be recovered from $u^\Sigma$ up to some $w_i^\mathrm{ker} \in \mathrm{ker}(\mathcal{A}_i^*)$. Indeed, the computation of $u_i = \mathcal{A}_i^*(w_i^\mathrm{ker} + w_i^\mathrm{im}) = \mathcal{A}_i^*w_i^\mathrm{im}$ remains unaffected. In this case we say that a subproblem has intrinsic redundancies.

The crucial insight is that, theoretically, we can extract all search directions $\{u_i\}_{i=1}^{N^\mathrm{dir}}$ for the individual subproblems from $u^\Sigma$. 
Thus, a learned algorithm which is solely build around $u^\Sigma$ holds the same information as a learned algorithm which makes use of $\{u_i\}_{i=1}^{N^\mathrm{dir}}$. In this case, these two learned algorithms could in principle be seen as equally powerful. This means if infinitely powerful networks are used (a infinitely powerful network can represent any deterministic relation exactly), it does not matter if we provide search directions or just the gradient of the standard data consistency term, i.e $\mathrm{NN}(u^\Sigma) \longleftrightarrow \mathrm{NN}(\{u_i\}_i)$.

In practice, individual subproblems can have individual inexactnesses, and individually encoded motion artefacts. Therefore, decomposing can be necessary for adequate treatment. The task of decomposing $u^\Sigma$ into $\{u_i\}_{i=1}^{N^\mathrm{dir}}$ is, however, highly non-trivial. For the case of MRI, for instance, such a decomposition would require to compute a Fourier transform, identify the correct fractions of the k-space and to compute an inverse Fourier transform. This is a non-trivial task for which most networks are not powerful enough in practice. Hence, despite the above observation providing $\{u_i\}_{i=1}^{N^\mathrm{dir}}$ will be a benefit in training speed, and reachable accuracy and reliability for the reconstruction. This is exactly what we will observe in our results in Section \ref{section:Results}.

\subsubsection{Dependent subproblems}
In most practical applications, subproblems will \emph{not} be independent. Rigorously, we can define redundancies between subproblems via linear dependencies of suboperators. 
For example in case of parallel MRI a densely sampled part of the k-space will introduce dependent subproblems. This further applies to most CT acquisition protocols (more severely the finer the data resolution), to MPI with multi-frame measurements, to general time-dependent acquisitions with oversampling, etc. In light of subsection \ref{subsec:ReSeSOp_with_additional_reg} even the introduction of an additional regularizer $R$ can introduce dependencies. Hence, a unique decomposition of $u^\Sigma$ into $\{u_i\}_{i=1}^{N^\mathrm{dir}}$ is no more feasible. In other words: by providing only $u^\Sigma$, information can be cancelled out. Especially in case of motion, details and sharp edges can be lost. By providing all relevant search directions $\{u_i\}_{i=1}^{N^\mathrm{dir}}$ our proposed learned ReSeSOp algorithm has access to more information translated to the image space. Also in this case, we will see emphasised benefits of learned ReSeSOp in our numerical results in Section \ref{section:Results}.

\begin{remark}
The issue of linear dependencies (and hence redundancies) can be illustrated mathematically in terms of the singular value decomposition (SVD). If lines of the operator $\mathcal{A} = U \Sigma V^*$ are linearly dependent, then the diagonal matrix $\Sigma$ contains zeros. In that case the approximate inverse can be expressed as $\mathcal{A}^+ = V \Sigma^+ U^*$, where
\begin{equation*}
\Sigma^+ = \left(\begin{array}{cccc|c}
\frac{1}{\sigma_1} & 0 & \cdots & 0  & \\
0 & \ddots & \ddots &  \vdots & {\mbox{\large $0_{K,M-K}$}}\\
\vdots & \ddots & \ddots & 0 & \\ 
0 & \dots & 0 & \frac{1}{\sigma_K} & \\&&&&\\\hline&&&&\\ &  {\mbox{\large $0_{N-K, K}$}}  & & & {\mbox{\large $0_{N-K, M-K}$}}
\end{array}\right) \in \R ^{N \times M}
\end{equation*}
with the zero matrices $0_{p,q}\in\R^{p \times q}$ and $\mathrm{rank}(\mathcal{A}) = K$. If the rows of $\mathcal{A}$ are linearly dependent, applying the pseudoinverse to $u^\Sigma$ as in the proof of lemma \ref{lem:lin_indep_rows} provides approximations 
\begin{equation}
\Bar{u}_i = \mathcal{A}_i^* U_i \Sigma^+ V^* u^\Sigma \neq u_i . \nonumber
\end{equation}
The error can be expressed as
\begin{align}
\|u_i - \Bar{u}_i \| =& 
\left\| \mathcal{A}_i^* U_i I_M U^* (\mathcal{A}s - y) - \mathcal{A}_i^* U_i (\Sigma^+)^\intercal V^* V \Sigma^\intercal U^* (\mathcal{A}s - y) \right\| \nonumber \\ 
=& \left\| \mathcal{A}_i^* U_i \begin{pmatrix}
0 & 0 \\
0 & I_{M-K}
\end{pmatrix} U^* (\mathcal{A}s - y) \right\| \nonumber \\ 
\leq & \underbrace{\left\| \mathcal{A}_i^* U_i \begin{pmatrix}
	0 & 0 \\
	0 & I_{M-K}
\end{pmatrix} U^* \right\|}_{B_i} \| (\mathcal{A}s - y)  \|, \nonumber
\end{align}
where $U_i$ are the rows of $U$ corresponding to suboperator $\mathcal{A}_i$. The matrix norms $B_i$ could actually be computed for a given problem and a given set of subproblems, and hence could be used to assess the necessity to use subproblems. 

The problem size in practical applications, however, might prevent actual computations. For the computation of a SVD, the full forward matrix has to be available. For instance in MRI with $34$ receiver coils and a resolution of $384\times384$ pixel the subsampled cartesian forward matrix has a memory footprint of approximately $540$GB. Recall, that the fast Fourier transform is responsible that the full forward matrix never has to be computed for reconstruction purposes. For the CT operator, we can compute the matrix norm $B_0$ for a few setups to illustrate redundancies, shown in table \ref{tab:B_i_values}. The value of $B_i$ scales with the size of subproblem $i$. A dimensionless or relative indicator could be $\frac{B_i}{\|\mathcal{A}_i\|}$, also presented in table \ref{tab:B_i_values}.
This value can be interpreted as a measure of irrecoverable information, if inexactnesses are severely distinct.

\begin{table}[]
\centering
\begin{tabular}{l||cccc|cccccc|cc}

\hline
\hline
$N_\alpha$ & 80 & 80 & 80 & 80 & 160 & 160 & 160 & 160 & 160 & 160 & 80 & 80 \\
$N_\mathrm{dtc}$ & 100 & 100 & 100 & 100 & 100 & 100 & 100 & 100 & 100 & 100 & 200 & 200 \\
$\alpha_\mathrm{max}$ & $\pi$ & $\pi$ & $\frac{5}{4}\pi$ & $\frac{5}{4}\pi$ & $\pi$ & $\pi$ & $\pi$ & $\frac{5}{4}\pi$ & $\frac{5}{4}\pi$ & $\frac{5}{4}\pi$ & $\pi$ & $\pi$ \\
$N^\mathrm{dir}$ & 4 & 5 & 4 & 5 & 4 & 5 & 40 & 4 & 5 & 40 & 4 & 5 \\
$B_0$ & 7.07  & 7.07 & 266 & 266 & 160 & 149 & 92.1 & 385 & 381 & 147 & 173 & 160 \\
$\|\mathcal{A}_0\|$ & 421 & 378 & 419 & 376 & 594 & 534 & 197 & 592 & 532 & 197 & 595 & 535 \\
$\frac{B_0}{\|\mathcal{A}_0\|}$ & \color{Green}0.02
& \color{Green}0.02
& \color{Red}0.63
& \color{Red}0.71
& \color{Orange}0.27
& \color{Orange}0.28
& \color{Orange}0.47
& \color{Red}0.65
& \color{Red}0.72
& \color{Red}0.75
& \color{Orange}0.29
& \color{Orange}0.30 \\

\hline
\hline
\end{tabular}
\caption{Values of matrix norm $B_0$ for a variety of CT-setups with a resolution of $100\times100$ pixels. Subsampled CT has negligible redundancies (\color{Green}green\color{black}). Oversampling in number of angles ($N_\alpha$) or detector positions ($N_\mathrm{dtc}$) introduces redundancies(\color{Orange}orange\color{black}). If more than 180 degrees are sampled ($\alpha_\mathrm{max}>\pi$), redundancies are severe (\color{Red}red\color{black}) and splitting into subproblems is highly recommended.}
\label{tab:B_i_values}
\end{table}

\end{remark}

\section{Results}\label{section:Results}

We investigate the performance of the learned ReSeSOp algorithm on various dynamic test cases from different imaging modalities, more precisely 
\begin{itemize}
\item on a simulated CT acquisition of a rheological flow through porous media (test case 1),
\item on the fastMRI dataset\cite{zbontar2019,knoll2020} where we add synthetically rigid motion and simulate a $4\times$ accelerated CARTESIAN k-space trajectory (test case 2), as well as a golden angle RADIAL acquisition (test case 3),
\item on real-time cardiac MRI measurements using a spiral k-space trajectory intrinsically containing non-rigid motion (test case 4). 
\end{itemize}
Thus, our test cases range from simulated data and motion to real dynamic measurements, they include rigid as well as highly non-linear deformations (test cases 1 and 4) and they comprise data with varying degrees of redundancy (for instance small redundancies in CT versus strong redundancies in radially sampled MRI).

We compare our proposed learned ReSeSOp algorithm to the learned primal (LP) network \cite{adler2018} and GANs \cite{usman2020}. To emphasise the structural benefits of including search directions instead of the standard data consistency gradient, we want to make the comparison as fair as possible. Hence, the primal network as well as the generator of the GAN are chosen exactly equal to the U-Net part of our learned ReSeSOp network. Also the iteration depth of learned primal and learned ReSeSOp is chosen equally. In some cases, the adversarial training of the GAN did not lead to satisfactory results. In such cases, we hence provide results of a U-Net trained with conventional non-adversarial loss function.

\subsection{The simulated rheological flow CT case (test case 1)}

We simulate the CT acquisition of a rheological stationary flow in porous media. The fluid is assumed to be incompressible. The Euler equations hence reduce to
\begin{equation}
\frac{\dd v_x}{\dd r_x} = - \frac{\dd v_y}{\dd r_y}, \nonumber 
\end{equation}
where $v = (v_x, v_y)^\intercal$ with $v_x, v_y: \mathbb{R}^2 \rightarrow \mathbb{R}$ is the flow field and $(r_x, r_y)$ parametrises the domain $\Omega$. We set the inflow and outflow to $1$ and employ Neumann conditions $v_x' = v_y' = 0$ for other boundaries. 

The data acquisition is simulated using the Radon transform 
as introduced in equation \eqref{eq:CT_forward} with 406 receiver points and 192 angles equally distributed over 180 degrees. The solution is computed on a $288\times288$ regular grid. The Radon transform is discretised via bilinear interpolation.

The	results are quantitatively presented in table \ref{tab:reco_quality_rheoFlowCT} and visually in figure \ref{fig:reconstructions_rheoFlowCT}. We see that learned ReSeSOp provides a reconstruction with sharp contours and without motion blur as occurs in the static FBP reconstruction. In light of the other test cases and the initialization of our learned ReSeSOp algorithm, we also included the static CG reconstruction.

Since this is a simulated test case, the ground truth image as well as the deformation fields are available. This means that in this test case, we can compare the performance of our learned ReSeSOp and the classical ReSeSOp algorithm from \cite{blanke2020} with the ground truth estimates $\mathcal{E}_i$ according to equation \eqref{eq:inexactnesses}. Since our learned ReSeSOp is based on considering subproblems based on the time variable (i.e. each index $i$ corresponds to one time instance), we use the same subproblem structure within the classical ReSeSOp algorithm. Figure 2 d) reveals that even with ground truth estimates for the time-dependent model inexactness at hand, classic ReSeSOp can only provide a blurred reconstruction. The reason is that the deformation in this test case is highly non-linear, i.e. the good motion compensation properties of classic ReSeSOp as demonstrated in \cite{blanke2020} would require to split the problem $\mathcal{A}s=y$ into subproblems depending on time and detector points which would result in high computational costs. By learning the stepsizes $\{\kappa_i^{(k)}\}_{i}$ for each iteration step along with the additional regularization $R^{(k)}$, our approach significantly reduces this computational complexity while providing a sharp and motion artefact free reconstruction. This highlights the efficacy of our learned ReSeSOp algorithm.

Furthermore, figure \ref{fig:reconstructions_rheoFlowCT} shows the performance of learned ReSeSOp in comparison to other learning based reconstruction methods. We note that the considered CT data acquisition and motion result in subtle redundancies. Hence, according to our discussion in Section \ref{subsec:ReSeSOp_w_wo_redundncies}, we expect a slight improvement by learned ReSeSOp since it explicitly exploits individual subproblems and hence a respective decomposition of the gradient of the full data discrepancy functional. Indeed, we observe that learned ReSeSOp provides visually and quantitatively better results.

\newcommand{\specialcell}[2][c]{%
\begin{tabular}[#1]{@{}c@{}}#2\end{tabular}}

\begin{table}
\centering
\begin{tabular}{ l||l|l|l|l|l|l } 
\hline
\hline

Method & FBP & CG & \specialcell{classical \\ ReSeSOp} & U-Net & learned primal & \specialcell{learned \\ ReSeSOp} \\
\hline
\hline

SSIM         & 0.3844 & 0.3947 & 0.2688 & 0.7946 & 0.8525 & \textbf{0.8701} \\
PSNR         & 14.32  & 14.61  & 14.44  & 18.45  & 20.02  & \textbf{20.56}  \\
MSE          & 27652  & 25880  & 26827  & 10634  & 7444   & \textbf{6569}   \\

\hline
\hline
\end{tabular}
\caption{Reconstruction quality of rheological flow CT case with 406 receiver points and 192 angles distributed over 180 degrees. Reconstructions of resolution $288\times288$, $N^\mathrm{dir} = 16$}
\label{tab:reco_quality_rheoFlowCT}
\end{table}

\begin{figure}
	\centering
	\captionsetup[subfigure]{margin=-3cm}
	\begin{subfigure}[b]{0.3\textwidth}
		\begin{tikzpicture}[spy using outlines={white,magnification=2,size=2.4cm, connect spies}]
			\node {\pgfimage[height=1.0\textwidth]{figures/learned_ReSeSOp_rheoFlow_ground_truth.png}};
			\spy on (1.1,0.6) in node [left] at (5,1.3);
		\end{tikzpicture}
		\caption*{a) ground truth}
	\end{subfigure}
	\hspace{0.05cm}
	\begin{subfigure}[b]{0.3\textwidth}
		\centering
		\begin{tikzpicture}[spy using outlines={black,magnification=2,size=2.4cm, connect spies}]
			\node {\pgfimage[height=1.0\textwidth]{figures/learned_ReSeSOp_rheoFlow_flow_thicker_cut.png}};
			\spy on (1.1,0.6) in node [left] at (-2.6,-1.3);
		\end{tikzpicture}
		\caption*{$\qquad \qquad \qquad \qquad \qquad \qquad \qquad \qquad$b) stationary flow field}
	\end{subfigure}\hspace{2.1cm} \hfill
	\\
	\begin{subfigure}[b]{0.3\textwidth}
		\begin{tikzpicture}[spy using outlines={white,magnification=2,size=2.4cm, connect spies}]
			\node {\pgfimage[height=1.0\textwidth]{figures/learned_ReSeSOp_rheoFlow_FBP.png}};
			\spy on (1.1,0.6) in node [left] at (5,1.3);
		\end{tikzpicture}
		\caption*{c) static FBP}
	\end{subfigure}
	\hspace{0.05cm}
	\begin{subfigure}[b]{0.3\textwidth}
		\centering
		\begin{tikzpicture}[spy using outlines={white,magnification=2,size=2.4cm, connect spies}]
			\node {\pgfimage[height=1.0\textwidth]{figures/learned_ReSeSOp_rheoFlow_CG.png}};
			\spy on (1.1,0.6) in node [left] at (-2.6,-1.3);
		\end{tikzpicture}
		\caption*{$\qquad \qquad \qquad \qquad \qquad \qquad \qquad \qquad$d) static CG }
	\end{subfigure}\hspace{2.1cm} \hfill
	\\
	\begin{subfigure}[b]{0.3\textwidth}
		\begin{tikzpicture}[spy using outlines={white,magnification=2,size=2.4cm, connect spies}]
			\node {\pgfimage[height=1.0\textwidth]{figures/learned_ReSeSOp_rheoFlow_classicReSeSOp.png}};
			\spy on (1.1,0.6) in node [left] at (5,1.3);
		\end{tikzpicture}
		\caption*{e) classical ReSeSOp }
	\end{subfigure}
	\hspace{0.05cm}
	\begin{subfigure}[b]{0.3\textwidth}
		\centering
		\begin{tikzpicture}[spy using outlines={white,magnification=2,size=2.4cm, connect spies}]
			\node {\pgfimage[height=1.0\textwidth]{figures/learned_ReSeSOp_rheoFlow_UNet.png}};
			\spy on (1.1,0.6) in node [left] at (-2.6,-1.3);
		\end{tikzpicture}
		\caption*{$\qquad \qquad \qquad \qquad \qquad \qquad \qquad \qquad$ f) U-Net}
	\end{subfigure}\hspace{2.1cm} \hfill
	\\
	\begin{subfigure}[b]{0.3\textwidth}
		\begin{tikzpicture}[spy using outlines={white,magnification=2,size=2.4cm, connect spies}]
			\node {\pgfimage[height=1.0\textwidth]{figures/learned_ReSeSOp_rheoFlow_LP.png}};
			\spy on (1.1,0.6) in node [left] at (5,1.3);
		\end{tikzpicture}
		\caption*{g) learned primal}
	\end{subfigure}
	\hspace{0.05cm}
	\begin{subfigure}[b]{0.3\textwidth}
		\centering
		\begin{tikzpicture}[spy using outlines={white,magnification=2,size=2.4cm, connect spies}]
			\node {\pgfimage[height=1.0\textwidth]{figures/learned_ReSeSOp_rheoFlow_LR_new.png}};
			\spy on (1.1,0.6) in node [left] at (-2.6,-1.3);
		\end{tikzpicture}
		\caption*{$\qquad \qquad \qquad \qquad \qquad \qquad \qquad \qquad$h) learned ReSeSOp}
	\end{subfigure}\hspace{2cm} \hfill
	
	\caption{Comparison of rheological Flow CT case. Resolution: $288 \times 288$, $N^\mathrm{dir} = 16$, 192 angles distributed over 180 degrees and 406 receiver points.}
	\label{fig:reconstructions_rheoFlowCT}
\end{figure}
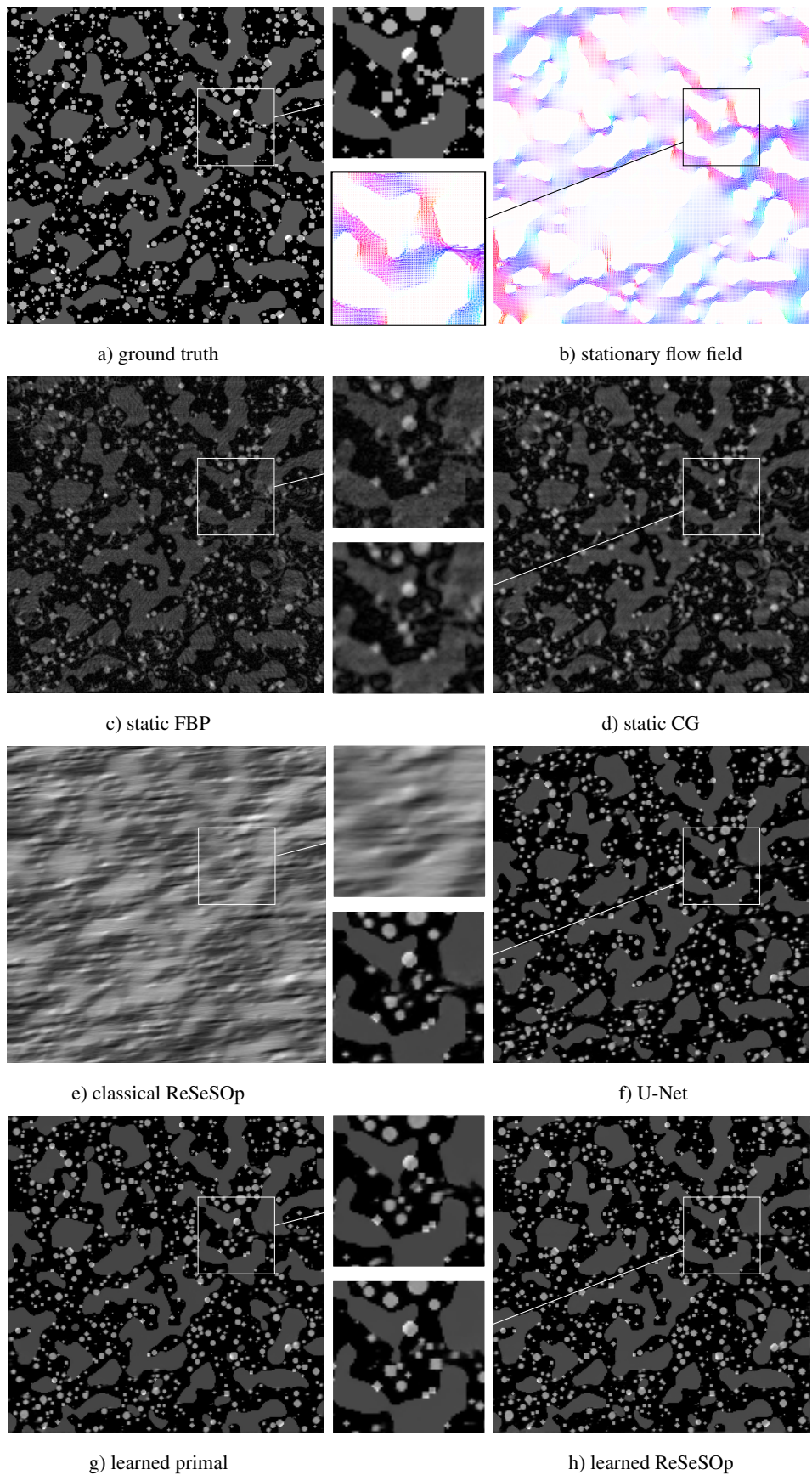

\subsection{The dynamic brain MRI case}

Next, we use the fastMRI data set which, however, does not yet contain motion corrupted data. Thus, we synthetically add rigid motion to each acquired brain, cf. \cite{feinler2024}. For this purpose we first compute accurate coil sensitivities using Espirit \cite{uecker2014}. Using these coil sensitivities we compute a complex-valued TV reconstruction from the static datasets. Using deformation fields, we can deform these reconstructions, apply the forward operator \eqref{eq:MRI_forward} and add measurement noise. Subsequently, we can subsample the generated data by using the standard $4\times$ regular cartesian subsampling, or, using a non-uniform Fast Fourier transforms (NUFFT), we can even simulate non-cartesian sampling schemes like Radial or Spiral acquisitions. Due to multiple receiver coils, the operator features redundancies in fully sampled regions. In case of radial acquisitions, the oversampling in k-space centre produce strong redundancies. This emphasises the necessity to employ learned ReSeSOp.

\subsubsection{CARTESIAN subsampling with non-uniformly distributed motion model (test case 2)}
Empirically, we do not know if and in which direction the subject will rotate or shift during the measurement. An appropriate probabilistic description of this scenario is a uniformly distributed random rigid deformation. To counter motion artefacts, motion estimation algorithms are usually utilised. However, the quality of estimates can significantly vary. The accuracy predominantly depends on the k-space position of the respective data fraction. Hence, the remaining inexactness is, probabilistically speaking, not equally distributed. We therefore imitate the result of a motion estimation algorithm for cartesian sequences. In this scenario, motion estimation algorithms manage to estimate motion particularly well on fully sampled regions in the centre of k-space, but (usually) fail for subsampled regions in the periphery of k-space, see e.g. \cite{feinler2024, rizzuti2022, cordero2016}. We model the extend of motion (remaining after motion estimation) by $u_x, u_y \in k_p \cdot \mathcal{U}(-8,8)$, $k_p = 0.1 + |p|$, where $p\in (-1,1)$ represents the position in k-space with $p=0$ denoting the centre of k-space, and $\alpha \in k_p \cdot \mathcal{U}(-6,6)$ degrees. In the following, we refer to this deformation as \emph{non-uniformly distributed}, a visualization is included in figure \ref{fig:ux_uy_alpha_over_i} b).

\begin{figure}
\centering
\begin{subfigure}[b]{0.45\textwidth}
\begin{tikzpicture}
\begin{axis}[
tick align=outside,
tick pos=left,
height=5cm,width = \textwidth, x grid style={darkgray176}, xlabel=$i$,
xtick style={color=black}, y grid style={darkgray176}, xmin=1, xmax=15, ymin=-8, ymax=8, xtick={2,4,6,8,10,12,14},
ytick style={color=black}, legend style={at={(0.675,0.0)}, anchor=south east, draw=none, fill=none}
]
\addplot [thick, Cyan]
table {
	1 -4.8765697
	2 -5.959059
	3 0.10938048
	4 -5.321519
	5 -1.2153745
	6 -0.5159986
	7 -2.8231802
	8 0.0
	9 -0.71831775
	10 -3.8025465
	11 -4.951071
	12 -0.6342201
	13 0.5817342
	14 -3.609818
	15 -0.53585434
};
\addplot [thick, Magenta]
table {
	1 -0.7812115
	2 2.4187663
	3 0.93331516
	4 3.293572
	5 0.11657746
	6 -0.92280185
	7 -1.2358662
	8 0.0
	9 3.238035
	10 3.0219238
	11 3.9334722
	12 3.9874337
	13 0.51513195
	14 1.3352209
	15 -3.641048 
};
\addplot [thick, YellowOrange]
table {
	1 1.2129245
	2 -0.6663921
	3 -2.2568185
	4 -3.729623
	5 -3.3692107
	6 1.4515619
	7 -2.0383325
	8 0.0
	9 -0.6869024
	10 -1.3552427
	11 -1.4911263
	12 1.887475
	13 -1.0432199
	14 -0.1462609
	15 0.9168919
};
\legend{$u_x$, $u_y$, $\alpha$}
\end{axis}
\end{tikzpicture}
\caption{uniformly distributed}
\end{subfigure}
\begin{subfigure}[b]{0.45\textwidth}
\begin{tikzpicture}
\begin{axis}[
tick align=outside,
tick pos=left,
height=5cm,width = \textwidth, x grid style={darkgray176}, xlabel=$i$,
xtick style={color=black}, y grid style={darkgray176}, xmin=1, xmax=15, ymin=-8, ymax=8, xtick={2,4,6,8,10,12,14},
ytick style={color=black}, legend style={at={(0.675,0.0)}, draw=none, fill=none, anchor=south east}
]
\addplot [thick, Cyan]
table {  
	1 -3.831692
	2 5.179036
	3 3.346246
	4 4.3318725
	5 -2.105997
	6 1.3380637
	7 0.3730418
	8 0.0
	9 -0.61205703
	10 -1.591821
	11 -0.49832338
	12 0.81513315
	13 1.4252806
	14 4.6252856
	15 -3.3502166
};
\addplot [thick, Magenta]
table {
	1 -6.6381073
	2 0.7382556
	3 7.1181
	4 -2.3353214
	5 5.170974
	6 2.1249824
	7 0.4566533
	8 0.0
	9 2.3288448
	10 -1.4805346
	11 2.6195774
	12 2.0040436
	13 6.7294893
	14 7.2899494
	15 -6.259389
};
\addplot [thick, YellowOrange]
table { 
	1 -0.7117014
	2 1.3318703
	3 4.7579494
	4 0.89422214
	5 4.3112946
	6 3.5837579
	7 0.93993276
	8 0.0
	9 1.1788621
	10 1.5894072
	11 2.8227277
	12 2.3795452
	13 -2.3199346
	14 -0.15363841
	15 -3.2819774
};
\legend{$u_x$, $u_y$, $\alpha$}
\end{axis}

\end{tikzpicture}
\caption{non-uniformly distributed}
\end{subfigure}
\caption{Visualization of one representative set of parameters for \emph{uniformly} and \emph{non-uniformly} distributed deformation. The deformation at $i=8$ is designed as the reference configuration. Consequently, we only visualize $\alpha_i - \alpha_8$, and analogously for other parameters.}
\label{fig:ux_uy_alpha_over_i}
\end{figure}
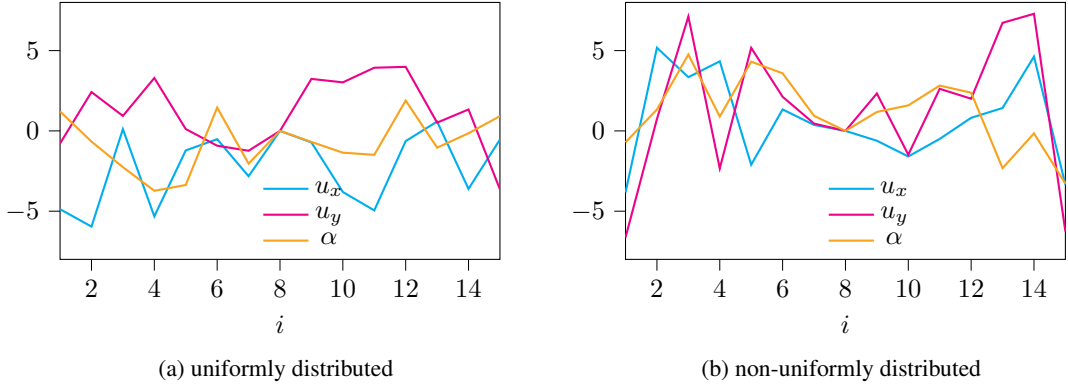

\newcommand{\appropto}{\mathrel{\vcenter{
\offinterlineskip\halign{\hfil$##$\cr
\propto\cr\noalign{\kern2pt}\sim\cr\noalign{\kern-2pt}}}}}

The accelerated CARTESIAN trajectory features redundancies due to the densely sampled k-space centre. Therefore, learned ReSeSOp has access to slightly more information compared to learned primal. This is confirmed by the qualitative results, shown in figure \ref{fig:reconstructions_dynamic_cartesian} and the quantitative results, shown in table \ref{tab:reco_quality}. Since the predictions made by GAN are consistently inferior to those of learned primal, we have excluded it from the qualitative results. Learned ReSeSOp converges in less epochs and to a slightly more accurate reconstruction than learned primal. Details are better recovered and results are more reliable indicated by overall lower error metrics on the test data set. Further, the reconstruction satisfies the ReSeSOp solution space far more accurately i.e. $\|w_i(s^\mathrm{pred})\| \approx \mathcal{E}_i$ for all $i=1,\ldots,N^\mathrm{dir}$, shown in figure \ref{fig:delta_over_N_U}. In contrast, the prediction of learned primal approximately satisfies $\|w_i(s^\mathrm{pred})\| \appropto \|y_i\|$ ignoring individual inexactnesses. 

For completeness, we include cartesian sampling with the uniformly distributed deformation i.e. without motion estimation, in our quantitative results.

\subsubsection{RADIAL subsampling with uniformly distributed motion model (test case 3)}
For the radial sampling case, we analogously assume that motion can occur uniformly over time. Even if motion estimation algorithms are used, prediction errors are approximately uniformly distributed in this case. Hence, we can always describe (remaining) motion as a random rigid deformation for each subproblem, uniformly distributed in time. I.e the object can be shifted by $u_x, u_y \in \mathcal{U}(-4,4)$ pixels and rotated by $\alpha \in \mathcal{U}(-3,3)$ degrees. We refer to this deformation as \emph{uniformly distributed}, a visualization is shown in figure \ref{fig:ux_uy_alpha_over_i} a).

The RADIAL trajectory provides significant redundancies. By including individual search directions with respect to all subproblems, learned ReSeSOp can selectively weight data contributions according to the ReSeSOp solution space. Qualitative results are shown in figure \ref{fig:reconstructions_dynamic_radial}, quantitative results are shown in table \ref{tab:reco_quality}. The plots in figure \ref{fig:delta_over_N_U} illustrate that our learned ReSeSOp approach provides solutions satisfying the ReSeSOp solution space with high accuracy in comparison to the learned primal method. The ground truth inexactnesses $\mathcal{E}_i$ are implicitly well approximated by the reconstruction $s^\mathrm{pred}$.

It is known, that trained medical doctors can \emph{read} from motion affected reconstructions. We expect that the performance of these doctors is similar to the performance of learned primal or image based methods like GANs. The learned ReSeSOp algorithm surpasses this bound without the need for explicit motion estimation, which is typically computationally expensive.

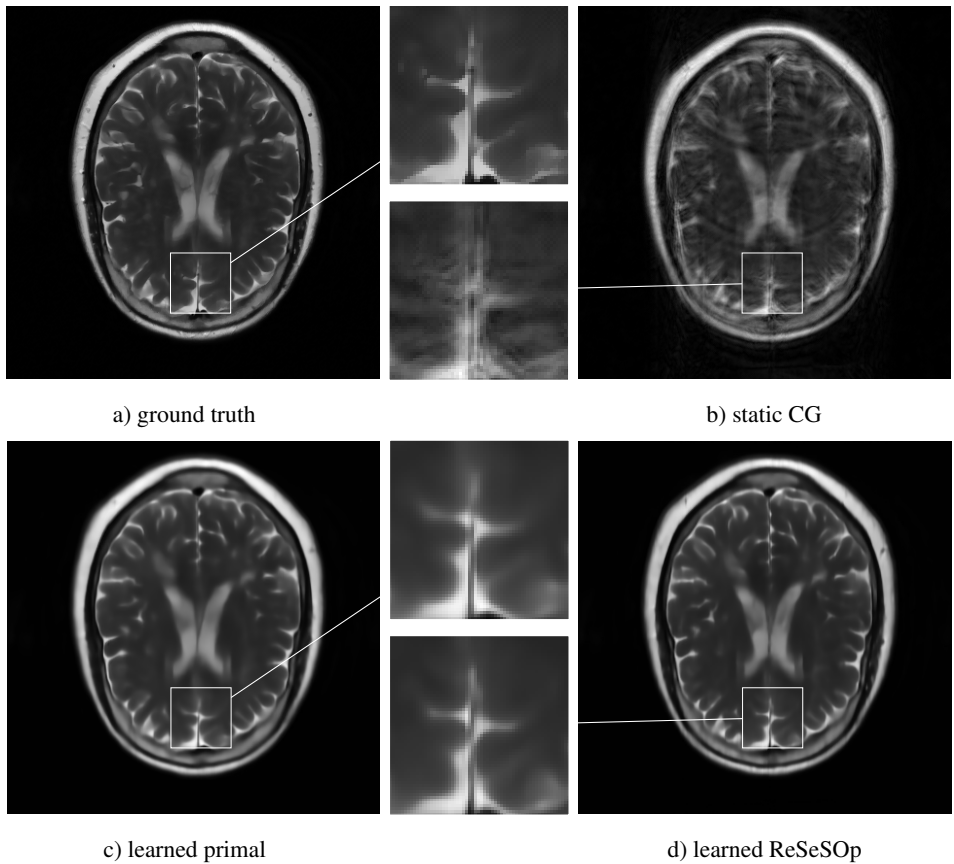
\begin{figure}
	\centering
	\captionsetup[subfigure]{margin=-3cm}
	
	\begin{subfigure}[b]{0.3\textwidth}
		\begin{tikzpicture}[spy using outlines={white,magnification=3,size=2.4cm, connect spies}]
			\node {\pgfimage[height=1.0\textwidth]{figures/learned_ReSeSOp_dynamic_CART_GT.png}};
			\spy on (0.1,-1.2) in node [left] at (5,1.3);
		\end{tikzpicture}
		\caption*{a) ground truth}
	\end{subfigure}
	\hspace{0.05cm}
	\begin{subfigure}[b]{0.3\textwidth}
		\centering
		\begin{tikzpicture}[spy using outlines={white,magnification=3,size=2.4cm, connect spies}]
			\node {\pgfimage[height=1.0\textwidth]{figures/learned_ReSeSOp_dynamic_CART_CG.png}};
			\spy on (0.1,-1.2) in node [left] at (-2.6,-1.3);
		\end{tikzpicture}
		\caption*{$\qquad \qquad \qquad \qquad\qquad \qquad \qquad \qquad$b) static CG}
	\end{subfigure}\hspace{2.1cm} \hfill
	\\
	\begin{subfigure}[b]{0.3\textwidth}
		\begin{tikzpicture}[spy using outlines={white,magnification=3,size=2.4cm, connect spies}]
			\node {\pgfimage[height=1.0\textwidth]{figures/learned_ReSeSOp_dynamic_CART_LP.png}};
			\spy on (0.1,-1.2) in node [left] at (5,1.3);
		\end{tikzpicture}
		\caption*{c) learned primal}
	\end{subfigure}
	\hspace{0.05cm}
	\begin{subfigure}[b]{0.3\textwidth}
		\centering
		\begin{tikzpicture}[spy using outlines={white,magnification=3,size=2.4cm, connect spies}]
			\node {\pgfimage[height=1.0\textwidth]{figures/learned_ReSeSOp_dynamic_CART_LReSeSOp.png}};
			\spy on (0.1,-1.2) in node [left] at (-2.6,-1.3);
		\end{tikzpicture}
		\caption*{$\qquad \qquad \qquad \qquad\qquad \qquad \qquad \qquad$d) learned ReSeSOp}
	\end{subfigure}\hspace{2cm} \hfill
	
	\caption{Comparison on dynamic case with CARTESIAN trajectory of learned ReSeSOp, CG and learned primal to the ground truth. Resolution : $384\times384$, $N^\mathrm{dir}=15, N^\mathrm{coils} = 15$. non-uniformly distributed deformations}
	\label{fig:reconstructions_dynamic_cartesian}
\end{figure}

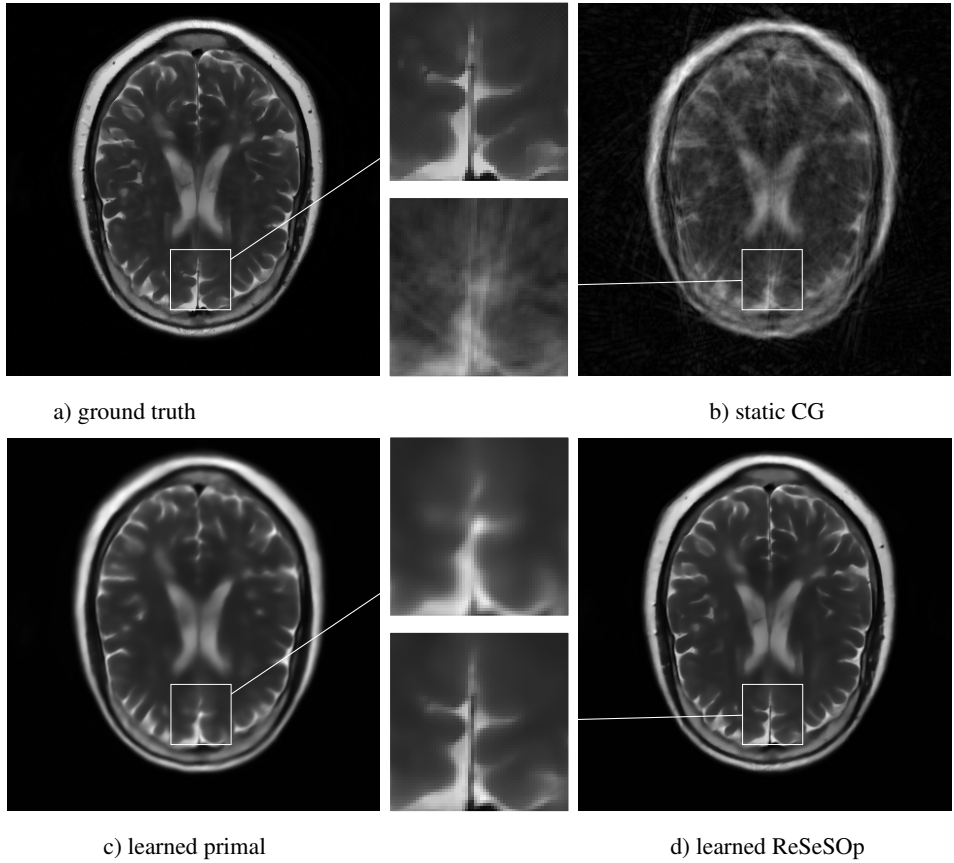
\begin{figure}
	\centering
	\captionsetup[subfigure]{margin=-3cm}
	\begin{subfigure}[b]{0.3\textwidth}
		\begin{tikzpicture}[spy using outlines={white,magnification=3,size=2.4cm, connect spies}]
			\node {\pgfimage[height=1.0\textwidth]{figures/learned_ReSeSOp_dynamic_RADIAL_GT.png}};
			\spy on (0.1,-1.2) in node [left] at (5,1.3);
		\end{tikzpicture}
		\caption*{a) ground truth$\qquad \qquad \quad$}
	\end{subfigure}
	\hspace{0.05cm}
	\begin{subfigure}[b]{0.3\textwidth}
		\centering
		\begin{tikzpicture}[spy using outlines={white,magnification=3,size=2.4cm, connect spies}]
			\node {\pgfimage[height=1.0\textwidth]{figures/learned_ReSeSOp_dynamic_RADIAL_CG.png}};
			\spy on (0.1,-1.2) in node [left] at (-2.6,-1.3);
		\end{tikzpicture}
		\caption*{$\qquad \qquad \qquad \qquad \qquad \qquad \qquad \qquad $ b) static CG}
	\end{subfigure}\hspace{2.1cm} \hfill
	\\
	\begin{subfigure}[b]{0.3\textwidth}
		\begin{tikzpicture}[spy using outlines={white,magnification=3,size=2.4cm, connect spies}]
			\node {\pgfimage[height=1.0\textwidth]{figures/learned_ReSeSOp_dynamic_RADIAL_LP.png}};
			\spy on (0.1,-1.2) in node [left] at (5,1.3);
		\end{tikzpicture}
		\caption*{c) learned primal}
	\end{subfigure}
	\hspace{0.05cm}
	\begin{subfigure}[b]{0.3\textwidth}
		\centering
		\begin{tikzpicture}[spy using outlines={white,magnification=3,size=2.4cm, connect spies}]
			\node {\pgfimage[height=1.0\textwidth]{figures/learned_ReSeSOp_dynamic_RADIAL_LReSeSOp.png}};
			\spy on (0.1,-1.2) in node [left] at (-2.6,-1.3);
		\end{tikzpicture}
		\caption*{$\qquad \qquad \qquad \qquad \qquad \qquad \qquad \qquad $ d) learned ReSeSOp}
	\end{subfigure}\hspace{2cm} \hfill
	
	\caption{Comparison on dynamic case with RADIAL trajectory of learned ReSeSOp, CG and learned primal to the ground truth. Resolution : $384\times384$, $N^\mathrm{dir}=15, N^\mathrm{coils} = 15$. uniformly distributed deformations}
	\label{fig:reconstructions_dynamic_radial}
\end{figure}

\begin{table}
\centering
\begin{tabular}{ l||l|l|l|l } 
\hline
\hline

Method & CG & learned primal & learned ReSeSOp & GAN \\
\hline
\hline
\multicolumn{5}{c}{CARTESIAN, $4\times$ accelerated, $N^\mathrm{dir}=15$, non-uniformly} \\
\hline

SSIM         & 0.7979 & 0.9496 & \textbf{0.9541} & 0.9100 \\
PSNR         & 28.07 & 33.68 & \textbf{34.561} & 29.90\\
MSE          & 320.1 & 71.07 & \textbf{60.345} & 191.4\\

\hline
\multicolumn{5}{c}{CARTESIAN, $4\times$ accelerated, $N^\mathrm{dir}=15$, uniformly} \\
\hline

SSIM         & 0.7432 & 0.9306 & \textbf{0.9436} & 0.8851 \\
PSNR         & 25.89 & 31.54 & \textbf{32.82} & 28.40 \\
MSE          & 543.6 & 120.6 & \textbf{89.14} & 255.43 \\



\hline
\multicolumn{5}{c}{RADIAL, 180 spokes, golden angle, $N^\mathrm{dir}=15$, uniformly} \\
\hline

SSIM         & 0.6724 & 0.8770 & \textbf{0.9540} & 0.8452\\
PSNR         & 25.93 & 27.69 & \textbf{35.01} & 26.14\\
MSE          & 487.1 & 312.3 & \textbf{56.13} & 467.1\\

\hline
\hline
\end{tabular}
\caption{Reconstruction quality of dynamic brain MRI case for CARTESIAN and RADIAL sampling with moderate and strong redundancies, respectively. Resolution of reconstruction: $384\times 384$ pixel.}
\label{tab:reco_quality}
\end{table}

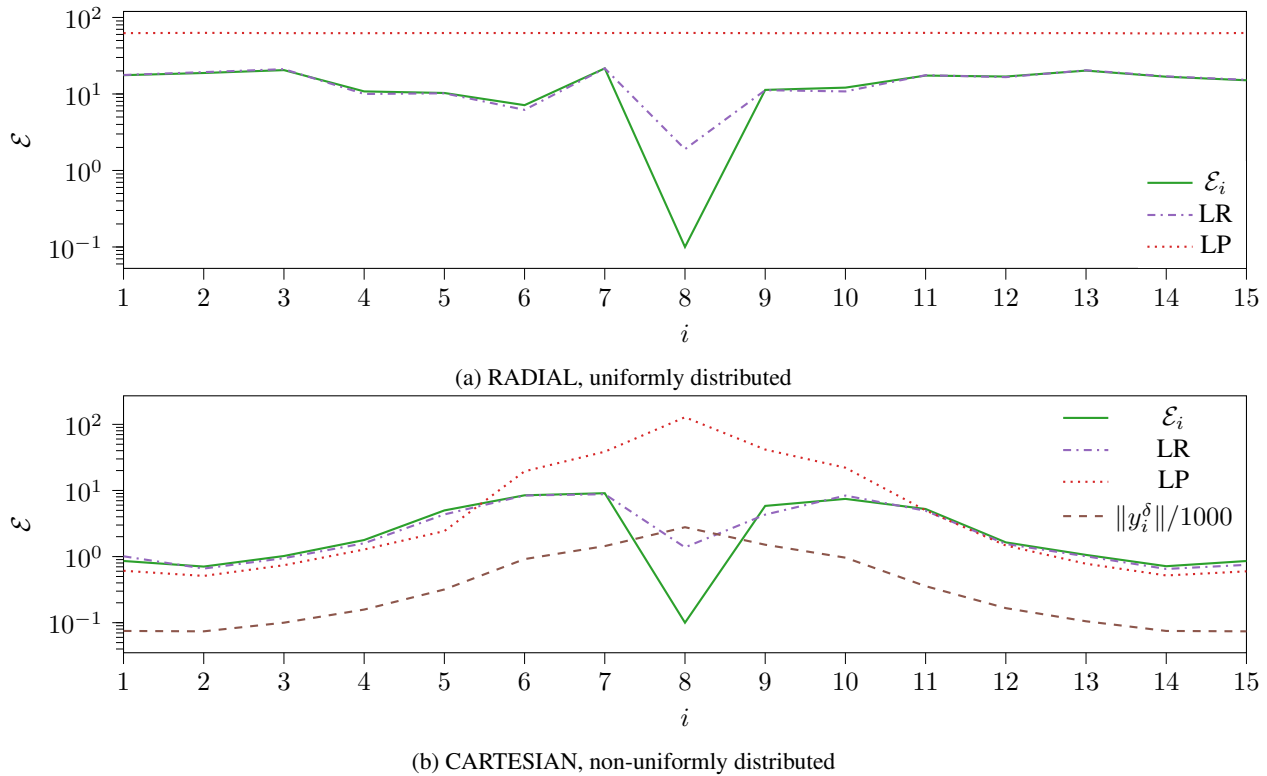
\begin{figure}
\centering
\begin{subfigure}[b]{\textwidth}
\begin{tikzpicture}

\definecolor{crimson2143940}{RGB}{214,39,40}
\definecolor{darkgray176}{RGB}{176,176,176}
\definecolor{darkorange25512714}{RGB}{255,127,14}
\definecolor{darkturquoise23190207}{RGB}{23,190,207}
\definecolor{forestgreen4416044}{RGB}{44,160,44}
\definecolor{goldenrod18818934}{RGB}{188,189,34}
\definecolor{gray127}{RGB}{127,127,127}
\definecolor{mediumpurple148103189}{RGB}{148,103,189}
\definecolor{orchid227119194}{RGB}{227,119,194}
\definecolor{sienna1408675}{RGB}{140,86,75}
\definecolor{steelblue31119180}{RGB}{31,119,180}

\begin{semilogyaxis}[
tick align=outside,
tick pos=left,
height=5cm,width = \textwidth, x grid style={darkgray176}, xlabel=$i$,
xtick style={color=black}, y grid style={darkgray176}, ylabel=$\mathcal{E}$, xmin=1, xmax=15,
ytick style={color=black}, legend style={at={(1,0.0)}, anchor=south east, draw=none}
]
\addplot [thick, forestgreen4416044]
table {
	1 17.6
	2 18.8
	3 20.5
	4 10.8
	5 10.3
	6 7.14
	7 21.6
	8 0.1
	9 11.3
	10 12.1
	11 17.4
	12 16.9
	13 20.2
	14 16.8
	15 15.1
};
\addplot [thick, dashdotted, mediumpurple148103189]
table {
	1 17.7
	2 19.3
	3 21.1
	4 10.0
	5 10.2
	6 6.2
	7 21.6
	8 1.9
	9 11.2
	10 10.8
	11 17.6
	12 16.6
	13 20.4
	14 17.0
	15 15.3
};
\addplot [thick, dotted, crimson2143940]
table {
	1 62.4
	2 62.9
	3 62.4
	4 62.3
	5 62.6
	6 62.6
	7 62.6
	8 62.8
	9 62.3
	10 62.4
	11 62.9
	12 62.3
	13 62.5
	14 61.7
	15 62.7
};

\legend{$\mathcal{E}_i$, LR, LP}
\end{semilogyaxis}

\end{tikzpicture}
\caption{RADIAL, uniformly distributed}

\end{subfigure}

\begin{subfigure}[b]{\textwidth}
\begin{tikzpicture}

\definecolor{crimson2143940}{RGB}{214,39,40}
\definecolor{darkgray176}{RGB}{176,176,176}
\definecolor{darkorange25512714}{RGB}{255,127,14}
\definecolor{darkturquoise23190207}{RGB}{23,190,207}
\definecolor{forestgreen4416044}{RGB}{44,160,44}
\definecolor{goldenrod18818934}{RGB}{188,189,34}
\definecolor{gray127}{RGB}{127,127,127}
\definecolor{mediumpurple148103189}{RGB}{148,103,189}
\definecolor{orchid227119194}{RGB}{227,119,194}
\definecolor{sienna1408675}{RGB}{140,86,75}
\definecolor{steelblue31119180}{RGB}{31,119,180}

\begin{semilogyaxis}[
tick align=outside,
tick pos=left,
height=5cm,width = \textwidth, x grid style={darkgray176}, xlabel=$i$,
xtick style={color=black}, y grid style={darkgray176}, ylabel=$\mathcal{E}$, xmin=1, xmax=15,
ytick style={color=black}, legend style={at={(1,1)}, draw=none, fill=none}
]
\addplot [thick, forestgreen4416044]
table {
	1 0.858
	2 0.706
	3 1.02
	4 1.78
	5 4.99
	6 8.45
	7 9.09
	8 0.1
	9 5.84
	10 7.46
	11 5.24
	12 1.64
	13 1.06
	14 0.715
	15 0.857
};
\addplot [thick, dashdotted, mediumpurple148103189]
table {
	1 1.01
	2 0.659
	3 0.946
	4 1.59
	5 4.34
	6 8.40
	7 8.78
	8 1.38
	9 4.31
	10 8.37
	11 4.93
	12 1.55
	13 1.01
	14 0.65
	15 0.75
};
\addplot [thick, dotted, crimson2143940]
table {
	1 0.61
	2 0.51
	3 0.74
	4 1.28
	5 2.43
	6 19.5
	7 38.6
	8 128.2
	9 41.6
	10 22.1
	11 5.00
	12 1.47
	13 0.779
	14 0.518
	15 0.598
};
\addplot [thick, dashed, sienna1408675]
table {
	1 0.075
	2 0.074
	3 0.10
	4 0.158
	5 0.318
	6 0.916
	7 1.441
	8 2.781
	9 1.511
	10 0.963
	11 0.357
	12 0.166
	13 0.105
	14 0.075
	15 0.074
};
\legend{$\mathcal{E}_i$, LR, LP, $\|y^\delta_i\|/1000$}
\end{semilogyaxis}

\end{tikzpicture}
\caption{CARTESIAN, non-uniformly distributed}
\end{subfigure}
\caption{Comparison of data consistency for dynamic brain MRI case with learned ReSeSOp (LR) and learned primal (LP) to the ground truth $\mathcal{E}_i$. We set the reference at $i=8$, i.e. $\mathcal{E}_8 = 0$ in case of no measurement noise. Due to visualization issues with log-scales we plot $\mathcal{E}_8 = 0.1$. }
\label{fig:delta_over_N_U}
\end{figure}

\subsection{The real-time cardiac MRI case (test case 4)} 
We are equipped with data of real-time cardiac MRI using an undersampled spiral k-space trajectory, provided to us by the authors of \cite{kleineisel2022}. In \cite{kleineisel2022}, the authors first computed a fully sampled reference using 40 spiral arms. The fully sampled data is hereby achieved by binning continuously acquired data over time. The authors proposed to use a variational network for reconstruction. The training was performed on real time measured data, leaving only 10 spiral arms per bin. This allows even free-breathing acquisitions. The variational network was able to compensate subsampling artefacts and produced reconstructions with good quality. The comparison to a low rank plus sparse method \cite{eirich2021} however indicates that it is beneficial to add temporally neighbouring arms and assign an appropriate number of subproblems to eliminate remaining in-frame motion smoothness.

This setup perfectly suits the application of learned ReSeSOp. Our results are shown in table \ref{tab:reco_quality_cardiacMRI} and in figure \ref{fig:reconstructions_cardiacMRI}. We see that the proposed method can indeed make use of the neighbouring arms to improve the overall reconstruction quality without compromising sharpness. The best result is achieved by learned ReSeSOp using 20 spiral arms and 4 subproblems. The results further show, that learned primal cannot split the gradient into contributions per subproblems sufficiently well. The reconstruction quality is severely reduced if 20 spiral arms are used for learned primal. Visually, this manifests in reduced sharpness of the heart. If, however, only 10 spiral arms are used, learned ReSeSOp and learned primal can be seen as quasi equivalent. Since the ground truth is calculated by binning, these 10 spiral arms are \emph{equally weighted} for the computation of the ground truth. Hence, no benefits from using multiple subproblems below the temporal resolution of ground truths can be expected. 

The authors in \cite{kleineisel2022} stated, that an even finer temporal resolution would be beneficial for a better determination of systolic and diastolic volumes of the heart. However, with only 5 spiral arms, the variational network was not able to compensate for subsampling artefacts. Using the learned ReSeSOp framework, the temporal resolution could indeed be doubled by using a sliding window approach and an appropriate number of subproblems (e.g. 5). For appropriate training, however, finer binned ground truths are required, and is therefore left for future work.

\begin{figure}
	\centering
	\captionsetup[subfigure]{margin=-3cm}
	\begin{subfigure}[b]{0.3\textwidth}
		\begin{tikzpicture}[spy using outlines={white,magnification=1.5,size=2.4cm, connect spies}]
			\node {\pgfimage[height=1.0\textwidth]{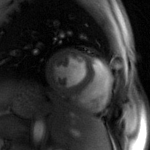}};
			\spy on (0.4,-0.4) in node [left] at (5,1.3);
		\end{tikzpicture}
		\caption*{a) Reference, 40 arms (binned)}
	\end{subfigure}
	\hspace{0.05cm}
	\begin{subfigure}[b]{0.3\textwidth}
		\centering
		\begin{tikzpicture}[spy using outlines={white,magnification=1.5,size=2.4cm, connect spies}]
			\node {\pgfimage[height=1.0\textwidth]{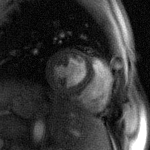}};
			\spy on (0.4,-0.4) in node [left] at (-2.6,-1.3);
		\end{tikzpicture}
		\caption*{$\qquad \qquad \qquad \qquad \qquad \qquad \qquad \qquad $b) static CG, 20 arms} 
	\end{subfigure}\hspace{2.1cm} \hfill
	\\
	\begin{subfigure}[b]{0.3\textwidth}
		\begin{tikzpicture}[spy using outlines={white,magnification=1.5,size=2.4cm, connect spies}]
			\node {\pgfimage[height=1.0\textwidth]{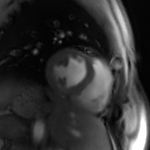}};
			\spy on (0.4,-0.4) in node [left] at (5,1.3);
		\end{tikzpicture}
		\caption*{c) learned primal, 20 arms} 
	\end{subfigure}
	\hspace{0.05cm}
	\begin{subfigure}[b]{0.3\textwidth}
		\centering
		\begin{tikzpicture}[spy using outlines={white,magnification=1.5,size=2.4cm, connect spies}]
			\node {\pgfimage[height=1.0\textwidth]{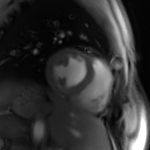}};
			\spy on (0.4,-0.4) in node [left] at (-2.6,-1.3);
		\end{tikzpicture}
		\caption*{$\qquad \qquad \qquad \qquad \qquad \qquad \qquad \qquad $d) learned ReSeSOp, 20 arms} 
	\end{subfigure}\hspace{2cm} \hfill
	
	\caption{Comparison on cardiac MRI case. Resolution: $448\times448$, $N^\mathrm{dir} = 4$, $N^\mathrm{coils} = 30-34$. Results are zoomed to the heart and brightened for better visualisation.}
	\label{fig:reconstructions_cardiacMRI}
\end{figure}

\begin{table}
\centering
\begin{tabular}{ l||l|l|l|l } 
\hline
\hline
Method & CG & learned primal & learned ReSeSOp & UNet \\
\hline
\hline

\multicolumn{5}{c}{10 arms} \\
\hline

SSIM         & 0.7812 & \textbf{0.9562} & 0.9549 & 0.9427 \\
PSNR         & 33.09 & 41.11 & \textbf{41.15} & 39.55 \\
MSE          & 28090 & \textbf{4391} & 4414 & 6207 \\

\hline
\multicolumn{5}{c}{20 arms} \\
\hline

SSIM         & 0.7866 & 0.9518 & \textbf{0.9582} & 0.9417 \\
PSNR         & 33.15 & 41.10 & \textbf{41.85} & 39.84 \\
MSE          & 26876 & 4360 & \textbf{3734} & 5861 \\

\hline
\hline
\end{tabular}
\caption{Reconstruction quality of spiral cardiac MRI with 40 arms reference. Each arm consists of 1408 k-space points. Resolution of reconstruction: $448\times448$ pixel. For learned ReSeSOp, we use 5 subproblems in case of 10 arms, and 4 subproblems in case of 20 arms.}
\label{tab:reco_quality_cardiacMRI}
\end{table}

\section{Conclusion}\label{section:Conclusion}

We presented a learned ReSeSOp algorithm which computes reconstructions and inexactnesses at the same time. In particular, our method (along with its results) is based on and in accordance with the theory of ReSeSOp. The numerical results reveal that the proposed method performs particularly well on motion affected data with redundant sampling schemes. The results clearly surpass the quality of comparable methods like GAN and the learned primal method. 

While we did not make use of any explicit motion estimation scheme, learned ReSeSOp can be used in connection with a preceding motion estimation procedure. The estimated deformations can be introduced into the forward model. Hence, learned ReSeSOp only has to correct for remaining motion artefacts, cf. test case 2. This can further increase accuracy and reliability of the reconstruction. 

Some reconstructed edges of learned ReSeSOp as well as learned primal are slightly too soft. This is especially the case for the accelerated cartesian sequences. This could be improved if these unrolled iterative networks are trained with an adversarial contribution c.f. \cite{usman2020,armanious2018b}. Such predictions are expected to satisfy the ReSeSOp solution space but contain sharper details. The ReSeSOp solution space is, however, non-degenerate and therefore can contain several nice looking images with different details. In that sense the network will be allowed to \emph{fantasise} within the ReSeSOp solution space. At first glance, this sounds like a disadvantage. However, note, that this space encloses the accurate description of data consistency containing the ground truth. This is indeed the optimal result given inexact forward operators.

\section*{Acknowledgments}

The work of the authors was funded by Bundesministerium f\"ur Bildung und Forschung (BMBF) under grant 05M2020 and by Deutsche Forschungsgemeinschaft (DFG, German Research Foundation) under Germany´s Excellence Strategy – EXC 2075 –390740016. The second author was further supported by the Sino-German Mobility Programme (M-0187) by the Sino-German Center for Research Promotion.

\bibliographystyle{unsrt}
\bibliography{Literatur}

\end{document}